\documentclass{amsart}
\usepackage{fullpage}

\usepackage{amsmath,amssymb,amsthm}
\usepackage[pdftex]{graphicx}
\usepackage{sidecap}
\usepackage{hyperref}
\usepackage{upgreek}
\usepackage{MnSymbol}

\newcommand{\pt}{\tau}
\newcommand{\st}{\uptau}
\newcommand{\gl}{\mathfrak{gl}}
\newcommand{\spn}{\mathfrak{sp}}
\newcommand{\df}{\zeta}
\newcommand{\dfe}{\zeta}
\DeclareMathOperator{\tr}{tr}
\DeclareMathOperator{\wt}{wt}
\DeclareMathOperator{\Tr}{tr}
\DeclareMathOperator{\Res}{Res}
\DeclareMathOperator{\zpois}{\mathfrak{z}_{\mathrm{Pois}}}
\DeclareMathOperator{\partition}{\vdash}
\DeclareMathOperator{\gr}{\mathrm{gr}}

\newtheorem{theorem}{Theorem}[section]
\newtheorem{lemma}{Lemma}[section]
\newtheorem{proposition}{Proposition}[section]
\newtheorem{conjecture}{Conjecture}[section]

\newtheorem*{otheorem}{Theorem}

\theoremstyle{definition}
\newtheorem{definition}{Definition}[section]

\theoremstyle{remark}
\newtheorem{remark}{Remark}[section]
\newtheorem{example}{Example}[section]

\begin{document}
\title{Representations of infinitesimal Cherednik algebras}
\author[F. Ding]{Fengning Ding}
\address{Phillips Academy, 180 Main St., Andover, MA 01810, USA}
\curraddr{Harvard College, Cambridge, MA 02138, USA}
\email{fding@college.harvard.edu}

\author[A. Tsymbaliuk]{Alexander Tsymbaliuk}
\address{Independent University of Moscow, 11 Bol'shoy Vlas'evskiy per., Moscow 119002, Russia}
\curraddr{Department of Mathematics, MIT, 77 Massachusetts Ave., Cambridge, MA  02139, USA}
\email{sasha\_ts@mit.edu}


\date{}
\begin{abstract}
Infinitesimal Cherednik algebras are continuous analogues of rational Cherednik algebras, and in the case of $\gl_n$, are deformations of universal enveloping algebras of the Lie algebras $\mathfrak{sl}_{n+1}$.
In the first half of this paper, 
we compute the determinant of the Shapovalov form, enabling us to classify all irreducible finite dimensional representations of $H_\df(\gl_n)$.
In the second half, we
investigate
Poisson-analogues of the infinitesimal Cherednik algebras
and generalize various results to $H_\dfe(\spn_{2n})$, including
Kostant's theorem.
\end{abstract}
\maketitle

\section*{Introduction}

The main goal of this paper is to study the representation theory
of infinitesimal Cherednik algebras $H_\df(\gl_n)$, a deformation
of the representation theory of $\mathfrak{sl}_{n+1}$ with infinitely
many deformation parameters $\df=(\df_0,\df_1,\df_2,...,\df_m,...)$. 
Namely, $\mathfrak{sl}_{n+1}$ can be
represented as $\gl_n\oplus V\oplus V^*$, where $V,V^*$ are the natural
representations of $\gl_n$ on vectors and covectors. 
In this representation of $\mathfrak{sl}_{n+1}$, the
elements of $V$ commute with each other, as do the elements of $V^*$.
The commutation relations of $\gl_n$ with $V,V^*$ are given by
the usual action of matrices on vectors and covectors, while commutators
of $V$ with $V^*$ produce elements of $\gl_n$. 
To pass to the deformation
$H_\df(\gl_n)$, one needs to change only the last relation: commutators of
$V$ and $V^*$ will now be not just elements of $\gl_n$ but rather some
polynomial $\df_0 r_0+\df_1 r_1 +\cdots$ of them, where $\df_i$ are the deformation parameters mentioned above
and $r_i$ are basis polynomials introduced in \cite{EGG}. 
This deformation turns out
to be very interesting, since it unifies the representation theory of
$\mathfrak{sl}_{n+1}$
with that of degenerate
affine Hecke algebras (\cite{D},\cite{L})
and of symplectic reflection algebras (\cite{EG}).

The main results of this paper are the following. 
In Section \ref{section:shapform}, we generalize a classical result from the representation theory of Kac-Moody algebras
by computing the determinant of the contravariant (or Shapovalov) form,
thus determining when the Verma module
over $H_\df(\gl_n)$ is irreducible. 
This proof requires knowledge of the quadratic central element and its action on the Verma module.
In Section \ref{section:gln}, we find
the quadratic central element of $H_\df(\gl_n)$;
this extends the work of Tikaradze \cite{T},
who proved using methods of homological algebra that the
center of $H_\df(\gl_n)$ is a polynomial algebra 
in $n$ generators, but did not get any explicit formulas for these generators. 
In Section \ref{section:finite},
we provide a complete classification and character
formulas for finite dimensional
representations of $H_\df(\gl_n)$, generalizing Chmutova's unpublished work.
In Sections \ref{section:poisson} to \ref{section:sp2n}, we introduce Poisson analogues of the infinitesimal Cherednik algebras, compute their Poisson center,
and use them to give a second proof of the formula for the quadratic central element of $H_\df(\gl_n)$.
We also provide an analogous formula for the center of the Poisson analogue of $H_\df(\spn_{2n})$.
Finally, in Section \ref{section:kostant},
we investigate an analogue of Kostant's theorem for $H_\df(\spn_{2n})$.

$\ $

\section{Basic Definitions}
		Let us formally define the infinitesimal Cherednik algebras of $\gl_n$, which we denote by $H_\df(\gl_n)$.
Let $V=\mathrm{span}(y_1,\ldots,y_n)$ be the basic $n$-dimensional representation of $\gl_n$ 
and $V^{*}=\mathrm{span}(x_1,\ldots,x_n)$ the dual representation.
For any $\gl_n$ invariant pairing $\df:V \times V^{*} \to U(\gl_n)$,
define an algebra $H_\df(\gl_n)$ as the quotient of the semi-direct product algebra 
$U(\gl_n)\ltimes T(V \oplus V^{*})$ 
by the relations $[y,x]=\df(y,x)$ and $[x,x']=[y,y']=0$
for all $x,x' \in V^{*}$ and $y,y' \in V$.
		
		Let us introduce an algebra filtration on $H_\df(\gl_n)$ by setting $\deg(x)=\deg(y)=1$ for $x \in V^*$,
$y \in V$, and $\deg (g)=0$ for $g  \in U(\gl_n)$.
We say that $H_\df(\gl_n)$ satisfies the PBW property if the natural surjective map 
$U(\gl_n)\ltimes S(V\oplus V^*) \twoheadrightarrow \mathrm{gr} H_\df(\gl_n)$ is an isomorphism, 
where $S$ denotes the symmetric algebra;
we call these $H_\df(\gl_n)$ the \emph{infinitesimal Cherednik algebras} of $\gl_n$.
In \cite{EGG}, Theorem 4.2, it was shown that the pairings $\df$
such that $H_\df(\gl_n)$ satisfy the PBW property are given by
$\df=\sum_{j=0}^k \df_j r_j$ where $\df_j \in \mathbb{C}$ and
$r_j$ is the symmetrization of the coefficient of $\tau^j$ in the expansion of
$(x,(1-\tau A)^{-1} y)\det (1-\tau A)^{-1}$.

		Note that for  
$\df=\df_0 r_0 + \df_1 r_1$ with $\df_1 \neq 0$, there is an isomorphism $\phi: H_\df(\gl_n) \rightarrow U(\mathfrak{sl}_{n+1})$ given by
$\phi(\alpha)= \alpha$ for $\alpha \in \mathfrak{sl}_n$,
$\phi(y_i)=\sqrt{\df_1} e_{i,n+1}$,
$\phi(x_i)=\sqrt{\df_1} \,e_{n+1,i}$, and
\[
\phi(\text{Id}) =\frac{1}{n+1} \left(e_{11}+\cdots+e_{nn}-n\,e_{n+1,n+1}-n\frac{\df_0}{\df_1}\right).
\]
This isomorphism allows us to view $H_\df(\gl_n)$ for general $\df$ as an interesting deformation 
of $U(\mathfrak{sl}_{n+1})$,
even though any formal deformation of $U(\mathfrak{sl}_{n+1})$ is trivial.

\begin{example} 
\label{example:gl1}
The infinitesimal Cherednik algebras of $\gl_1$ are generated by elements $e$, $f$, and $h$,
satisfying the relations $[h,e]=e$, $[h,f]=-f$, and $[e,f]=\phi(h)$
for some polynomial $\phi$. In literature, these algebras are known as generalized Weyl algebras (\cite{S}).
\end{example}

Similarly to the representation theory of $\mathfrak{sl}_{n+1}$, we define the Verma module of $H_\df(\gl_n)$ as
\[
	M(\lambda) =H_\df(\gl_n)/\{H_\df(\gl_n) \cdot \mathfrak{n}^++H_\df(\gl_n)(h-\lambda(h))\}_{h \in \mathfrak{h}}
\]
where the set of positive root elements $\mathfrak{n}^+$ is spanned by the positive root elements of $\gl_n$ 
(i.e., matrix units $e_{ij}$ with $i<j$) and elements of $V$; the set of negative root elements
$\mathfrak{n}^-$ is spanned by the negative root elements of $\gl_n$ (i.e., matrix units $e_{ij}$ with $i>j$) and elements of $V^*$; and the Cartan subalgebra $\mathfrak{h}$ is spanned 
by diagonal matrices. The highest weight, $\lambda$, is an element of $\mathfrak{h}^*$, and $v_\lambda$ is the corresponding
highest-weight vector.

	Let us denote the set of positive roots by $\Delta^{+}$, so that $\Delta^{+}= \{e_{ii}^*-e_{jj}^*\} \cup \{e_{kk}^*\}$ for $1\leq i<j\leq n$, $1 \leq k \leq n$.
To denote the positive roots of $\gl_n$, we use $\Delta^{+}\left(\gl_n\right)$, and
to denote the weights of $y_i$, we use $\Delta^{+}(V)$.
We define $\rho = \frac{1}{2} \sum_{\lambda \in \Delta^+(\gl_n)} \lambda=\left(\frac{n-1}{2},\frac{n-3}{2},\ldots,-\frac{n-1}{2}\right)$,
a \emph{quasiroot} to be an integral multiple of an element in
$\Delta^+$, and $Q^+$ to be the set of linear combinations of positive roots with nonnegative integer coefficients.
Finally, $U(\mathfrak{n}^-)_\nu$ denotes the $-\nu$ weight-space of $U(\mathfrak{n}^-)$, where $\nu \in Q^+$.

$\ $

\section{Shapovalov Form}
\label{section:shapform}
	\indent As in the classical representation theory of Lie algebras, the Shapovalov form can be used to investigate the basic structure of Verma modules. Similarly to the classical case, $M(\lambda)$ possesses a maximal proper submodule 
${\overline M}(\lambda)$ and has a unique irreducible quotient $L(\lambda)=M(\lambda)/{\overline M}(\lambda)$. 
Define the Harish-Chandra projection 
$\text{HC}: H_\df(\gl_n) \rightarrow S(\mathfrak{h})$ with respect to the decomposition	$H_\df(\gl_n)=(H_\df(\gl_n)\mathfrak{n}^{+}+\mathfrak{n}^{-} H_\df(\gl_n)) \oplus U(\mathfrak{h})$, and let
$\sigma:H_\df(\gl_n) \rightarrow H_\df(\gl_n)$ be the anti-involution that takes
$y_i$ to $x_i$ and $e_{ij}$ to $e_{ji}$.
	
	\begin{definition}
		The \emph{Shapovalov form} $S: H_\df(\gl_n)\times H_\df(\gl_n) \rightarrow U(\mathfrak{h})\cong S(\mathfrak{h})\cong \mathbb{C}[\mathfrak{h}^*]$
		is a bilinear form given by
		 $S(a,b)=\text{HC}(\sigma(a)b)$. 
		 The bilinear form $S(\lambda)$ on the Verma module $M(\lambda)$ is defined by $S(\lambda)(u_1 v_\lambda, u_2 v_\lambda)=S(u_1, u_2)(\lambda)$, for $u_1, u_2 \in U(\mathfrak{n}^-)$.
	\end{definition}
	This definition is motivated by the following two properties (compare with \cite{KK}):
	\begin{proposition}
		\label{proposition:shap}
		\normalfont{1.} $S(U(\mathfrak{n}^{-})_{\mu},U(\mathfrak{n}^{-})_{\nu})=0$\ \ for $\mu \neq \nu$,
		\\
		\indent\normalfont{2.} ${\overline M}(\lambda)=\ker S(\lambda)$.
	\end{proposition}
	
	Statement 1 of Proposition \ref{proposition:shap} reduces $S$ to its restriction to $U(\mathfrak{n}^-)_\nu \times U(\mathfrak{n}^-)_\nu$, 
	which we will denote as $S_\nu$.
	Statement 2 of Proposition \ref{proposition:shap} gives a necessary and sufficient condition for the Verma module $M(\lambda)$ to be irreducible, namely that for any $\nu \in Q^+$,
the bilinear form $S_\nu(\lambda)$ is nondegenerate, or equivalently, that $\det S_\nu (\lambda) \neq  0$, where the determinant is computed in any basis; note that this condition is independent of basis. 
For convenience, we choose the basis $\{f^{\mathbf{m}}\}$, where $\mathbf{m}$ runs over all partitions of $\nu$ into a sum of positive roots and $f^{\mathbf{m}}=\prod f_\alpha^{\mathbf{m}_\alpha}$ with $f_\alpha \in \mathfrak{n}^-$ of weight $-\alpha$. 
We will use the notation
$a \partition b$ to mean that $(a_1, \ldots, a_n)$ is a partition of $b$ into a sum of $n$ nonnegative integers when $b \in \mathbb{N}$,
and $\mathbf{m}\partition \nu $ to mean that $\mathbf{m}$ is a partition of $\nu$ into a sum of elements of $\Delta^+$ when $\nu \in Q^+$.
Then, the basis we will work with 
is $\{f^{\mathbf{m}}\}_{\mathbf{m} \partition \nu}$.
	
	Now, we present a formula for the determinant of the Shapovalov form for $H_\df(\gl_n)$
	generalizing the classical result presented in \cite{KK}.
	This formula uses 
	the following result proven
	in Section \ref{subsection:actiongln}:
	for a deformation $\df=\df_0 r_0 + \df_1 r_1 + \cdots + \df_m r_m$,
	the central element $t_1'$ (introduced in Section \ref{section:gln}) acts on the Verma module $M(\lambda)$ by 
	a constant
$	P(\lambda)=\sum_{j=0}^{m+1} w_j H_{j}(\lambda+\rho)$,
where
	$H_j(\lambda)=\sum_{p \partition j} \prod_{1\leq i \leq n} \lambda_i^{p_i}$
	are the complete symmetric functions (we take $H_0(\lambda)=1$)
	and $w_j(\df_0,\ldots,\df_j)$ are linearly independent linear functions on $\df_k$.
	
	Define the Kostant partition function $\tau$ as $\tau(\nu)=\dim U(\mathfrak{n}^-)_{\nu}$.
	Then:
	
	\begin{theorem}
	\label{theorem:shapovalov}
		Up to a nonzero constant factor, the Shapovalov determinant computed in the basis $\{f^{\mathbf{m}}\}_{\mathbf{m} \partition \nu}$ is given by
		\[
			\det S_\nu (\lambda) =\left( \prod_{\alpha \in \Delta^+(V)} \prod_{k=1}^{\infty}
			\left(P(\lambda)-P(\lambda-k\,\alpha) \right)^{\tau(\nu-k\,\alpha)} \right)
			\left(
				\prod_{\alpha \in \Delta^+(\gl_n)} \prod_{k=1}^{\infty}
				\left((\lambda+\rho,\alpha)-k\right)^{\tau(\nu-k\,\alpha)}
			\right).
		\]
	\end{theorem}
	\begin{remark} 
	In the case $\df=\df_0 r_0 + \df_1 r_1$ with $\df_1 \neq 0$, we get the classical formula from \cite{KK}.
	\end{remark}
	    
	\begin{proof}
		The proof of this theorem is quite similar to the classical case with a few technical details and differences 
		that will be explained below.
		We begin with the following lemma, which shows that irreducible factors of $\det S_\nu(\lambda)$ must divide
		$P(\lambda)-P(\lambda-\mu)$ for some $\mu \in Q^+$.
		\begin{lemma}
		\label{lemma:shap1}
			Suppose $\det S_\nu(\lambda)=0$. 
			Then, there exists $\mu\in Q^+\backslash \{0\}$ such that $P(\lambda)-P(\lambda-\mu)=0$.
		\end{lemma}
		\begin{proof}
			Note that $\det S_{\nu}(\lambda)=0$ 
			implies that the Verma module $M(\lambda)$ has a critical vector (a vector on which all elements of $\mathfrak{n}^+$ act by 0)
			of weight $\lambda-\mu$ for some $\mu \in Q^+$ satisfying $0 < \mu < \nu$. 
			Thus, $M(\lambda-\mu)$ is embedded in $M(\lambda)$.
			Since $t_1'$ acts by constants on both $M(\lambda)$ and $M(\lambda-\mu)$, which can be considered as a submodule of $M(\lambda)$, we get
			$P(\lambda)=P(\lambda-\mu)$.
		\end{proof}
		
		The top term of the Shapovalov determinant $\det S_\nu (\lambda)$ in the basis $\{f^{\mathbf{m}}\}_{\mathbf{m} \partition \nu}$
		comes from the product of diagonal elements, that is,
		$\prod_{\mathbf{m} \partition \nu} \prod [\sigma(f_{\alpha}),f_{\alpha}]^{\mathbf{m}_{\alpha}}(\lambda)$. 
		The top term of $[e_{ij},e_{ji}](\lambda)$ for $i<j$ is
		$\lambda_i-\lambda_j=(\lambda,\alpha)$ where $\alpha$ is the weight of $e_{ij}$. 
		The following lemma gives the top term of $[y_j,x_j](\lambda)$:
		\begin{lemma}
			The highest term of $[y_j,x_j](\lambda)$ for $\df=\df_0 r_0+\cdots+\df_m r_m$ is
			$\df_m \sum_{\mathbf{p}} (\mathbf{p}_j+1)\prod \lambda_i^{\mathbf{p}_i}$,
			where the sum is over all partitions $\mathbf{p}$ of $m$ into $n$ summands.
		\end{lemma}
		\begin{proof}
			From \cite{EGG}, Theorem 4.2, we know that the top term of
			$[y_j,x_j]$ for $\df=\df_0 r_0 +\df_1 r_1+\cdots+r_m$ is given by the coefficient of $\tau^m$ in $\det(1-\tau A)^{-1}(x_j, (1-\tau A)^{-1} y_j)$.
			Because the set of diagonalizable matrices is dense in $\gl_n$, 
			we can assume $A$ is a diagonal matrix $A=\text{diag}(\lambda_1,\lambda_2,...,\lambda_n)$
			so that \\*$\det(1-\tau A)^{-1}=\prod{\frac{1}{1-\tau \lambda_i}}=\sum_k \sum_{p \partition k} \prod_i \lambda_i^{\mathbf{p}_i} \tau^k$ 
			and
				$x_j (1-\tau A)^{-1} y_j=\frac{1}{1-\tau \lambda_j}=1+\lambda_j \tau+\cdots$.
			Multiplying these series gives the statement in the lemma.
		\end{proof}	
		Thus, we see that the top term of the determinant computed in the basis $\{f^{\mathbf{m}}\}_{\mathbf{m} \partition \nu}$, up to a scalar multiple, is of the form
		\[
			\left(\prod_{\alpha \in \Delta^+(\gl_n)} (\lambda,\alpha)^{\sum_{\mathbf{m}}\mathbf{m}_{\alpha}} 
			\right)
			\left(
			\prod_{\alpha=\wt(y_j) \in \Delta^+(V)}
			\left(\sum_{\mathbf{p}} (\mathbf{p}_j+1)\prod \lambda_i^{\mathbf{p}_i}\right)^{\sum_{\mathbf{m}} \mathbf{m}_\alpha}
			\right).
		\]
		Since $\tau(\mu)$ is the number of partitions of a weight $\mu$,
		the sum $\sum_\mathbf{m} \mathbf{m}_{\alpha}$ over all partitions $\mathbf{m}$ of $\nu$ with $\alpha$ fixed must equal 
		$\sum_{k=1}^{\infty} \tau(\nu-k\alpha)$,
		so the expression above simplifies to
		\begin{eqnarray*}
			\left(\prod_{\alpha \in \Delta^+(\gl_n)} \prod_{k=1}^{\infty} (\lambda,\alpha)^{\tau(\nu-k\alpha)} 
			\right)
			\left(
			\prod_{\alpha=\wt(y_j) \in \Delta^+(V)} \prod_{k=1}^{\infty} 
			\left(\sum_{\mathbf{p} \partition m} (\mathbf{p}_j+1)\prod \lambda_i^{\mathbf{p}_i}\right)^{\tau(\nu-k\alpha)}
			\right).
		\end{eqnarray*}
		
		This highest term comes from the product of the highest terms of factors of $P(\lambda)-P(\lambda-\mu)$
		for various $\mu \in Q^+$. 
		
		\begin{lemma}
		\label{lemma:irreducible}
		\normalfont{1.} For all $\mu \neq k\alpha$, $\alpha \in \Delta^+(\gl_n)$,
		$P(\lambda)-P(\lambda-\mu)$ is irreducible as a polynomial in $\lambda$.\\
		\indent\normalfont{2.} For $\mu=k \alpha$, $\alpha \in \Delta^+(\gl_n)$,
		$\frac{P(\lambda)-P(\lambda-k\alpha)}{(\lambda+\rho,\alpha)-k}$ is irreducible.
		\end{lemma}
		
		If Lemma \ref{lemma:irreducible} is true, then all $\mu$ contributing to the above product must be quasiroots:
		if $\mu \neq k\alpha$ for some $\alpha\in \Delta^+(\gl_n)$,
		the highest term of the irreducible polynomial
		$P(\lambda)-P(\lambda-\mu)$,
		$
			\sum_{p \partition m} \sum_{j} \mu_j (\mathbf{p}_j+1) \prod \lambda_i^{\mathbf{p}_i}
		$,
		does not match any factor in the highest term of the Shapovalov determinant unless
		$\mu$ is a $V$-quasiroot. 
		Moreover, if $\mu=k\alpha$ for $\alpha\in \Delta^+(\gl_n)$,
		 since 
		$\frac{P(\lambda)-P(\lambda-k\alpha)}{(\lambda+\rho,\alpha)-k}$ is irreducible
		for $\alpha \in \Delta^+(\gl_n)$, 
		comparison with the highest term of the determinant shows that only the linear factor $(\lambda+\rho,\alpha)-k$ of $P(\lambda)-P(\lambda-k\alpha)$ appears in the Shapovalov determinant.

		\begin{proof}
			We will prove that $P(\lambda)-P(\lambda-\mu)$ is irreducible for $\mu \neq k\alpha$ ($\alpha \in \Delta^+(\gl_n)$); similar arguments will show that $\frac{P(\lambda)-P(\lambda-k\alpha)}{(\lambda+\rho,\alpha)-k}$ is irreducible for any
		$\alpha \in \Delta^+(\gl_n)$, $k\in \mathbb{N}$.
			
		Consider the parameters $w_i$ as formal variables.
		Then, we have $P(\lambda)-P(\lambda-\mu)=\sum_{i \geq 0} w_i (H_{i}(\lambda+\rho)-H_{i}(\lambda+\rho-\mu))$.
		We can absorb the $\rho$ vector into the $\lambda$ vector. 
		For this polynomial to be reducible in $w_i$ and $\lambda_j$,
		the coefficient of $w_1$ should be zero:
		$H_1(\lambda)-H_1(\lambda-\mu)=H_1(\mu)=0$.
		Also, since the coefficient of $w_2$ is linear in $\lambda_j$, 
		it must divide the coefficients of every other $w_i$.
		In particular, the highest term of $H_2(\lambda)-H_2(\lambda-\mu)$ must divide that of $H_3(\lambda)-H_3(\lambda-\mu)$.
		The highest term of $H_2(\lambda)-H_2(\lambda-\mu)$ is $\sum_i \lambda_i (\mu_i+\sum_j \mu_j)=(\lambda,\mu)$
		and the highest term of $H_3(\lambda)-H_3(\lambda-\mu)$ is given by $H_3'(\lambda)(\mu)$,
		the evaluation of the gradient $H_3'(\lambda)$ at $\mu$. 
		Since this term is quadratic and is divisible by $(\lambda,\mu)$, we can write
		$H_3'(\lambda)(\mu)=(\lambda,\mu)(\lambda,\xi)$ for some $\xi \in \mathfrak{h}^*$.
		Now, let us match coefficients of $\lambda_i \lambda_j$ for $i\neq j$ and of $\lambda_i^2$ on both sides of the equation.
		By doing so (and using the fact that $\sum \mu_i =0$), we obtain $\mu_i \xi_j+\mu_j \xi_i=\mu_i+\mu_j$
		and $\mu_i \xi_i=2\mu_i$. Since $\mu_1+\cdots+\mu_n=0$ and $\mu \neq 0$,
		at least two of $\mu_i$ are nonzero, say $\mu_{i_1}$ and 
		$\mu_{i_2}$. From the two equations, we obtain $\mu_{i_1}+\mu_{i_2}=0$.
		If $\mu_{i_3}\neq 0$, then by similar arguments,
		$\mu_{i_1}+\mu_{i_3}=\mu_{i_2}+\mu_{i_3}=\mu_{i_1}+\mu_{i_2}=0$,
		which is impossible since $\mu_{i_1},\mu_{i_2},\mu_{i_3}\neq 0$.
		Thus, $P(\lambda)-P(\lambda-\mu)$ is reducible
		only if exactly two of the $\mu_i$ are nonzero and opposite to each other; that is, $\mu=k\alpha$ for 
		$\alpha \in \Delta^+(\gl_n)$.
		\end{proof}
				
		To prove that the power of each factor in the determinant formula of Theorem \ref{theorem:shapovalov} is correct, 
		we use an argument involving the Jantzen filtration, which we
		define as in \cite{KK}, page 101 (for our purposes, we switch $U(\mathfrak{g})$ to $H_\df(\gl_n)$). 
		The Jantzen filtration is a technique to track the order of zero of a bilinear form's determinant.
		Instead of working over the complex numbers, we consider the ring of localized polynomials $\mathbb{C} \langle t \rangle=\{\frac{p(t)}{q(t)} \ |\  p(t), q(t) \in \mathbb{C}[t], q(0) \neq 0\}$.
		A word-to-word generalization of \cite{KK}, Lemma 3.3, proves
		that the power of $P(\lambda)-P(\lambda-k\alpha)$ for $\alpha \in \Delta^+(V)$ and of $(\lambda+\rho,\alpha)-k$ for $\alpha \in \Delta^+(\gl_n)$ is given by $\tau(\nu-k\alpha)$, completing the proof of Theorem \ref{theorem:shapovalov}.
	\end{proof} 


$\ $

\section{The Casimir Element of $H_\df(\gl_n)$}
\label{section:gln}
	Let $\mathcal{Q}_1, \mathcal{Q}_2, \mathcal{Q}_3,...,\mathcal{Q}_n \in S(\gl_n^*)$ (which can be identified as elements of $S(\gl_n)$ under the trace-map) be defined
	by the power series
	$\det(t \text{Id} -X)=\sum_{j=0}^{n}{(-1)^j t^{n-j} \mathcal{Q}_j(X)}$,
	and let $\beta_i$ be the image of $\mathcal{Q}_i$ under the symmetrization map from $S(\gl_n)$ to $U(\gl_n)$.
	The center of $U(\gl_n)$ is a polynomial algebra generated by these $\beta_i$.
	Define $t_i=\sum_j x_j[\beta_i,y_j]$. 
	According to \cite{T}, Theorems 2.1 and 1.1,
	the center of $H_0(\gl_n)$
	is a polynomial algebra in $\{t_i\}_{1 \leq i \leq n}$,
	and there exist unique (up to a constant) $c_i \in \mathfrak{z}(U(\gl_n))$ such that the center of $H_\df(\gl_n)$
	is a polynomial algebra in $t_i'=t_i+c_i$, $1 \leq i \leq n$. 
	\begin{definition}
	 The \emph{Casimir element} of $H_\df (\gl_n)$ is defined (up to a constant) as $t_1'$.
	 \end{definition} 
	 We will construct the Casimir element of $H_\df(\gl_n)$ 
	and prove that its action on the Verma module $M(\lambda)$ is given by $P(\lambda)=\sum_{j=0}^{m+1} w_j H_{j}(\lambda+\rho)$, where
$w_j$ are linear functions in $\df_i$.

\subsection{Center}
	Let us switch to the approach elaborated in \cite{EGG}, Section 4, where all deformations satisfying the PBW property were determined.
	Define $\delta^{(m)}=(i\partial)^m \delta$ with $\delta$ being a standard delta function
	at 0, i.e., 
	$\int \delta(\theta) \phi(\theta) d\theta=\phi(0)$.
	Let $f(z)$
	be a polynomial satisfying
	$f(z)-f(z-1) =  \partial^n (z^n \df(z) )$,
	where $\df(z)$ is the generating series of the deformation parameters: $\df(z)=\df_0+\df_1 z + \df_2 z^2+\cdots$.
	Since $f(z)$ is defined up to a constant, we can specify $f(0)=0$.
	Recall from \cite{EGG}, Section 4.2, that for $\hat{f}(\theta)=\sum_{m \geq 0} f_m \delta^{(m)}(\theta)$,
	\[
		[y,x]=\frac{1}{2 \pi^n}\int_{v\in \mathbb{C}^n: |v|=1} (x,(v\otimes \bar{v})y)
		\int_{-\pi}^{\pi} \left(1-e^{-i \theta}\right)\hat{f}(\theta)e^{i\theta(v\otimes \bar{v})} \, d\theta \, dv.
	\]
	
	\begin{theorem}
	\label{theorem:firstcentralelementtheorem}
		Let $g(z)=\sum g_m z^m=\sum \frac{f_m}{(m+1)(m+2)\cdots(m+n-1)}z^m$.
		The Casimir element of $H_\df(\gl_n)$ is given by
		$t_1'=\sum x_j y_j + \mathrm{Res} _{z=0} g(z^{-1})\det\left(1- zA\right)^{-1} dz/z$.
	\end{theorem}
	
	\begin{proof}
	Define $C'=\mathrm{Res}_{z=0} g(z^{-1})\det\left(1- zA\right)^{-1} dz/z$.
	Let us compute 
	$[y,t_1+C']= \sum_j [y,x_j] y_j + [y,C']$.
	The first summand is:
	\begin{eqnarray*}
		\sum_j [y,x_j] y_j &=& \frac{1}{2 \pi^n}\sum_j \int_{v \in \mathbb{C}^n: |v|=1} 
		\int_{-\pi}^{\pi} (1-e^{-i\theta})\hat{f}(\theta) e^{i\theta (v\otimes \bar{v})} (x_j,(v\otimes \bar{v})y)y_j \, d\theta \, dv
		\\
		&=&  \frac{1}{2 \pi^n}
		\int_{|v|=1} 
		\int_{-\pi}^{\pi} (1-e^{-i\theta})\hat{f}(\theta) e^{i\theta (v\otimes \bar{v})} \otimes (v\otimes \bar{v})y \, d\theta \, dv.
	\end{eqnarray*}
	
	Following \cite{EGG}, Section 4.2, we define
	$F_m(A)=\int_{|v|=1} \langle Av,v \rangle ^{m+1}\, dv = \int_{|v|=1} (v \otimes \bar{v})^{m+1} \, dv$.
	There, it was proven that 
	\[\sum_m f_m F_{m-1}(A)=2 \pi^n\mathrm{Res}_{z=0}g(z^{-1})\det(1-zA)^{-1} z^{-1} dz=2 \pi^n C'.\]
	Thus, we can write
	\[
		C'=\frac{1}{2 \pi^n} \sum_m f_m \int_{|v|=1} (v \otimes \bar{v})^m dv=\frac{1}{2 \pi^n} \int_{|v|=1} \int_{-\pi}^{\pi} \hat{f}(\theta) e^{i\theta (v\otimes \bar{v})}\, d\theta \, dv,
	\]
	which implies that
	$
		[y,C']=\frac{1}{2 \pi^n} \int_{|v|=1} \int_{-\pi}^{\pi} \hat{f}(\theta) [y, e^{i\theta (v\otimes \bar{v})}] \, d\theta \, dv 
	$.
	Since
	\begin{center}
	$
		e^{-i\theta (v\otimes \bar{v})} [y,e^{i\theta (v\otimes \bar{v})}] =
		e^{-i\theta (v\otimes \bar{v})} ye^{i\theta (v\otimes \bar{v})} -y
		=
		e^{-i\theta \mathrm{ad}(v\otimes \bar{v})} y-y
		=
		(e^{-i\theta}-1)(v\otimes \bar{v})y
	$,
	\end{center}
	we get
	$
		[y,C']=\frac{1}{2 \pi^n} \int_{|v|=1} \int_{-\pi}^{\pi} \hat{f}(\theta) e^{i\theta (v\otimes \bar{v})}(e^{-i\theta}-1)(v\otimes \bar{v})y \, d\theta \, dv
	$,
	and so
	$\sum_i [y,x_i] y_i + [y,C'] =0$
	as desired.
	By using the anti-involution $\sigma$ defined in the beginning of Section \ref{section:shapform}, this implies $[x,t_1+C']=0$ for any $x \in V^*$, while $[e_{ij},t_1+C']=0$ by \cite{T},
	and hence, $t_1'=t_1+C'$.
	\end{proof}
	\begin{remark}
	This proof resembles calculations in \cite{EGG}, Section 4. 
	In particular, Proposition 5.3 of \cite{EGG} provides a formula for the Casimir element of continuous Cherednik algebras. 
	However, adopting this formula for the specific case of infinitesimal Cherednik algebras is nontrivial and requires the above
	computations.
	\end{remark}
	
\subsection{Action of the Casimir Element on the Verma Module}
	\label{subsection:actiongln}
	\indent In this section, we justify our claim that the action of the Casimir element $t_1'$ is given by $	P(\lambda)=\sum_{j=0}^{m+1} w_j H_{j}(\lambda+\rho)$.
	Obviously, $t_1'$ acts by a scalar on $M(\lambda-\rho)$, which
	we will denote by $t_1'(\lambda)$. Since 
	$t_1'=\sum x_i y_i + C'$,
	$C' \in \mathfrak{z}(U(\mathfrak{g}))\cong S(\mathfrak{g})^G$, we see that $t_1'(\lambda)=C'(\lambda)$ where $C'(\lambda)$ denotes the 
	constant by which $C'$ acts on $M(\lambda-\rho)$.
	
	\begin{theorem}
		\label{theorem:casimirgln}
		Let 
		$w(z)$ be the unique degree $m+1$ polynomial satisfying
		$f(z)=(2\sinh (\partial/2))^{n-1} z^{n-1} w(z)$.
		Then
		\[t_1'(\lambda)= \sum_{p \geq 0} w_p H_{p}(\lambda).\]
	\end{theorem}
	\begin{proof}
		
		Because $C'(\lambda)$ is a polynomial in $\lambda$, we can consider a finite-dimensional representation of $U(\gl_n)$ 
		instead of the Verma module $M(\lambda-\rho)$ of $H_\df(\gl_n)$.
		For a dominant weight $\lambda-\rho$ (so that the highest weight $\gl_n$-module $V_{\lambda-\rho}$ is finite dimensional) we define the normalized trace
		$T(\lambda, \theta)=\mathrm{tr}_{V_{\lambda-\rho}} (e^{i\theta(v\otimes \bar{v})})/\dim V_{\lambda-\rho}$
		for any $v$ satisfying $|v|=1$ (note that $T(\lambda, \theta)$ does not depend on $v$). 
		To compute $T(\lambda, \theta)$, we will use the Weyl Character formula (see \cite{Ful}): $\chi_{\lambda-\rho}=\frac{\sum_{w\in W} (-1)^w e^{w\lambda}}{\sum_{w\in W}(-1)^w e^{w\rho}}$,
		where $W$ denotes the Weyl group (which is $S_n$ for $\gl_n$).
		However, direct substitution of $e^{i\theta(v\otimes \bar{v})}$ into this formula gives zero in the denominator, so instead we compute
		 $\lim_{\epsilon \rightarrow 0} \chi_{\lambda-\rho}( e^{i\theta(v\otimes \bar{v})+\epsilon\mu})$ for a general
		diagonal matrix $\mu$.
		
		Without loss of generality, we may suppose $v=y_1$, so that
		\[v\otimes \bar{v}=q=
		\left(
\begin{matrix}
	1 & 0 & \cdots & 0\\
	0 & 0 & \cdots & 0\\
	\vdots & \vdots & \ddots & \vdots \\
	0 & 0 & \cdots & 0
\end{matrix}
\right)
.\]
 Then
		\begin{align*}
		\lim_{\epsilon \rightarrow 0} \chi_{\lambda-\rho}( e^{i\theta(v\otimes \bar{v})+\epsilon\mu})
		&=\lim_{\epsilon \rightarrow 0} \frac{\sum_{w\in S_n} (-1)^w e^{\langle w\lambda, i \theta q + \epsilon \mu \rangle}}{\sum_{w\in S_n}(-1)^w e^{ \langle w\rho, i \theta q + \epsilon \mu \rangle}}\\
		&=
		\lim_{\epsilon \rightarrow 0} \frac{\sum_{w\in S_n} (-1)^w e^{\langle w\lambda, i \theta q + \epsilon \mu \rangle}}{\prod_{\alpha \in  \Delta^+(\gl_n)}
		(e^{\langle \alpha/2, i \theta q + \epsilon \mu \rangle}-e^{-\langle \alpha/2, i \theta q + \epsilon \mu\rangle})
		}.
		\end{align*} 
		Partition $\Delta^+(\gl_n)$ into $\Delta_1 \sqcup \Delta_2=\Delta^+(\gl_n)$,
		where $\Delta_1=\{e_{11}^*-e_{jj}^*: 1<j\leq n\}$.
		For $\alpha\in \Delta_1$, 
		\[
		\lim_{\epsilon \rightarrow 0}
		\left(e^{\langle \alpha/2, i \theta q + \epsilon \mu \rangle}-e^{-\langle \alpha/2, i \theta q + \epsilon \mu \rangle}\right)
		=e^{i\theta/2}-e^{-i\theta/2}=2i\sin\left(\frac{\theta}{2}\right),
		\]
		so
		$
		\lim_{\epsilon \rightarrow 0}
		\prod_{\alpha \in  \Delta_1}
		(e^{\langle \alpha/2, i \theta q + \epsilon \mu \rangle}-e^{-\langle \alpha/2, i \theta q + \epsilon \mu \rangle})^{-1}
		=\left(2i\sin\left(\frac{\theta}{2}\right)\right)^{1-n}$.
		
		Next, we compute the numerator. We can divide $S_n=\bigsqcup_{1\leq j \leq n} B_j$, where $B_j=\{w \in S_n | w(j)=1 \}$.
		Note that $B_j=\sigma_j \cdot S_{n-1}$,
		where $\sigma_j=(12\cdots j)$ and $S_{n-1}$ denotes the subgroup of $S_n$ corresponding to permutations
		of $\{1,2,\ldots,j-1,j+1,\ldots,n\}$.
		We can then write
		\begin{align*}
		\sum_{w\in B_j}(-1)^w e^{\langle w\lambda, i  \theta q + \epsilon \mu \rangle}&=
		\sum_{\sigma \in S_{n-1}} (-1)^{\sigma_j} (-1)^{\sigma} e^{i\theta \lambda_j} e^{\epsilon \langle \sigma_j \circ \sigma(\lambda), \mu \rangle}
		=(-1)^{j-1} e^{i \theta \lambda_j} e^{\epsilon \lambda_j \mu_1}
		\sum_{\sigma \in S_{n-1}} (-1)^{\sigma} e^{\epsilon \langle \sigma(\tilde{\lambda}_{j}), \tilde{\mu} \rangle}
		\end{align*}
		where $\widetilde{\lambda_j}=(\lambda_1, \ldots, \lambda_{j-1}, \lambda_{j+1},\ldots, \lambda_n)$
		and $\tilde{\mu}=(\mu_2,\ldots,\mu_n)$.
		
		Combining the results of the last two paragraphs, we get
		\begin{align*}
		&\lim_{\epsilon \rightarrow 0} \frac{\sum_{w\in S_n} (-1)^w e^{\langle w\lambda, i \theta q + \epsilon \mu \rangle}}{\prod_{\alpha \in  \Delta^+(\gl_n)}
		(e^{\langle \alpha/2, i \theta q + \epsilon \mu \rangle}-e^{-\langle \alpha/2, i \theta q + \epsilon \mu \rangle})
		}\\
		&=\lim_{\epsilon \rightarrow 0}
		\sum_{1 \leq j \leq n}(-1)^{j-1} \frac{e^{i\theta \lambda_j+\epsilon \lambda_j\mu_1}}{(2i\sin\frac{\theta}{2})^{n-1}}
		\frac{\sum_{\sigma \in S_{n-1}} (-1)^{\sigma} e^{\epsilon \langle \sigma(\tilde{\lambda}_{j}), \tilde{\mu} \rangle} }{\prod_{\alpha \in  \Delta_2}
		(e^{\langle \alpha/2, i \theta q + \epsilon \mu \rangle}-e^{-\langle \alpha/2, i \theta q + \epsilon \mu \rangle})}.
		\end{align*}
		Using the Weyl character formula again, 
		we see that
		\[
		\frac{\sum_{\sigma \in S_{n-1}} (-1)^{\sigma} e^{\epsilon \langle \sigma(\tilde{\lambda}_{j}), \tilde{\mu} \rangle}}{\prod_{\alpha \in  \Delta_2}
		(e^{\langle \alpha/2, \epsilon \mu \rangle}-e^{-\langle \alpha/2, \epsilon \mu\rangle})}
		=
		\Tr_{V_{\tilde{\lambda}_j-\tilde{\rho}}}(e^{\epsilon \tilde{\mu}})
		\]
		where $\tilde{\rho}$ is half the sum of all positive roots of $\gl_{n-1}$.
		Thus,
		\[
		\lim_{\epsilon \rightarrow 0}
		\frac{\sum_{\sigma \in S_{n-1}} (-1)^{\sigma} e^{\epsilon \langle \sigma(\tilde{\lambda}_{j}), \tilde{\mu}\rangle}}{\prod_{\alpha \in  \Delta_2}
		(e^{\langle \alpha/2, i \theta q + \epsilon \mu \rangle}-e^{-\langle \alpha/2, i \theta q + \epsilon \mu \rangle})}
		=
		\Tr_{V_{\tilde{\lambda}_j-\tilde{\rho}}}(1)
		=
		\dim V_{\tilde{\lambda}_j-\tilde{\rho}}.
		\]
		We substitute to obtain
		\begin{align*}
		\Tr_{V_{\lambda-\rho}} (e^{i\theta(v\otimes \bar{v})})=
		\sum_{1 \leq j \leq n}(-1)^{j-1} \frac{e^{i\theta \lambda_j}\dim V_{\tilde{\lambda}_j-\tilde{\rho}}}{(2i\sin\frac{\theta}{2})^{n-1}}.
		\end{align*}
		Our original goal was to calculate
		$T(\lambda, \theta)=\Tr_{V_{\lambda-\rho}} (e^{i\theta(v\otimes \bar{v})})/\dim V_{\lambda-\rho}$.
		We obtain
		\[
		T(\lambda, \theta)=\sum_{1 \leq j \leq n}(-1)^{j-1} \frac{e^{i\theta \lambda_j}\dim V_{\tilde{\lambda}_j-\tilde{\rho}}}{(2i\sin\frac{\theta}{2})^{n-1} \dim V_{\lambda-\rho}}.
		\]
		Using the dimension formula (\cite{Ful}, Equation 15.17):
		\[
		\dim V_{\lambda-\rho}=\prod_{1 \leq i < j \leq n} \frac{\lambda_i-\lambda_j}{j-i},
		\]
		we get
		$T(\lambda,\theta)=(2 i \sin(\theta/2))^{1-n}(n-1)! \sum_{j=1}^{n} \frac{e^{i\lambda_j \theta}}{\prod_{k \neq j} (\lambda_j-\lambda_k)}$.
		Since $\sum_{j=1}^{n} \frac{x_j^m}{\prod_{k\neq j}(x_j-x_k)} = H_{m-n+1} (x_1, ..., x_n)$,
		we have
		\begin{center}
		$T(\lambda,\theta)=(2 i \sin(\theta/2))^{1-n}(n-1)! \sum_{p\geq 0} \frac{H_p(\lambda) (i\theta)^{p+n-1}}{(p+n-1)!}$.
		\end{center}
		
		Thus, we get 
		\begin{align*}
		t_1'(\lambda)&=C'(\lambda)=\left(\frac{1}{2 \pi^n} \int_{|v|=1}\!\! \int_{-\pi}^{\pi} \hat{f}(\theta) e^{i\theta (v\otimes \bar{v})} d\theta dv \right)(\lambda)=\frac{1}{(n-1)!}\int_{-\pi}^{\pi}\hat{f}(\theta) T(\lambda, \theta) d\theta  \\
							&=\int_{-\pi}^{\pi} \hat{f}(\theta) (2 i \sin(\theta/2))^{1-n} \sum_{p\geq 0} \frac{H_p(\lambda) (i\theta)^{p+n-1}}{(p+n-1)!} d\theta
							=
							\sum_{p \geq 0} w_p' H_{p}(\lambda),
		\end{align*}
		where
		$w_p'=\int_{-\pi}^{\pi} \hat{f}(\theta) (2 i \sin(\theta/2))^{1-n} \frac{ (i\theta)^{p+n-1}}{(p+n-1)!} d\theta$.
		Let $w'(z)=\sum w_p' z^p$.		
		We verify that
				\begin{align*}
			\left(e^{\partial/2}-e^{-\partial/2} \right)^{n-1}\!\! z^{n-1} w'(z)
					&=
						\int_{-\pi}^{\pi} \hat{f}(\theta) \sum_{p \geq 0} 
								(2 i \sin(\theta/2))^{1-n}\! \left(e^{\partial/2}-e^{-\partial/2} \right)^{n-1}\!\! \frac{ (iz\theta)^{p+n-1}}{(p+n-1)!}\, d\theta
								\\
					&=
						\int_{-\pi}^{\pi} \hat{f}(\theta)
								(2 i \sin(\theta/2))^{1-n}\! \left(e^{\partial/2}-e^{-\partial/2} \right)^{n-1}\!\! e^{iz\theta}\, d\theta
					\\
					&=
						\int_{-\pi}^{\pi} \hat{f}(\theta)
								(2 i \sin(\theta/2))^{1-n}\! \left(e^{i\theta/2}-e^{-i\theta/2} \right)^{n-1}\!\! e^{iz\theta} \, d\theta
					\\
					&=
						\int_{-\pi}^{\pi} \hat{f}(\theta) e^{iz\theta} \, d\theta
					= f(z),
		\end{align*}
		and it is easy to see that the polynomial solution to $f(z)=(2\sinh (\partial/2))^{n-1} z^{n-1} w(z)$ is unique.
	\end{proof}

$\ $

\section{Finite Dimensional Representations}
\label{section:finite} \indent In this section, we investigate when
the irreducible $H_\df(\gl_n)$ representation $L(\lambda)$ is finite
dimensional. As in the case of classical Lie algebras, any finite
dimensional irreducible representation is isomorphic to $L(\lambda)$
for a unique weight $\lambda$. Theorem~\ref{theorem:classification}
provides a necessary and sufficient condition for $L(\lambda)$ to be
finite dimensional. In particular, all such representations have a
\emph{rectangular form}.

In Section \ref{subsection:finiteexistence}, we prove that for any
allowed \emph{rectangular form} there exists a deformation $\df$
such that the representation $L(\lambda)$ of $H_\df(\gl_n)$ has exactly that shape.

\subsection{Rectangular Nature of Irreducible Representations}
\label{subsection:finitecharacterization}

\begin{theorem}
\label{theorem:classification} \normalfont{(a)} The representation $L(\lambda)$
is finite dimensional if and only if $\lambda$ is a dominant $\gl_n$
weight and there exists $\nu_n \in \mathbb{N}_0$ such that
$P(\lambda)=P(\lambda-(0, \ldots, 0, \nu_n+1))$.

 For every $1\leq i\leq n-1$ 
let $k_i\in \mathbb{N}_0$  be the smallest nonnegative integer such that
$P(\lambda)=P(\lambda-(0,\ldots,0,k_i+1,0,\ldots,0))$
(we set $k_i=\infty$ if no such nonnegative integer exists).
We define parameters $\nu_i = \min(k_i,\lambda_i-\lambda_{i+1})$.

\normalfont{(b)} If $L(\lambda)$ is finite dimensional, then as a $\gl_n$ module
it decomposes into
\[
L(\lambda)=\bigoplus_{0 \leq \lambda-\lambda' \leq \nu} V_{\lambda'},
\]

where $\nu=(\nu_1,\ldots,\nu_n)$ are the parameters defined above (depending on $\df$ and $\lambda$).
\end{theorem}

\begin{proof}
In order for $L(\lambda)$ to be finite dimensional,
it is clearly necessary for $\lambda$ to be a dominant $\gl_n$ weight. 
Recalling the PBW property and the definition
of the Verma module $M(\lambda)$, we see that as a $\gl_n$ module,
$M(\lambda)$ decomposes as $M(\lambda)= V_\lambda \oplus (V_\lambda \otimes
S_1) \oplus (V_\lambda \otimes S_2) \oplus \cdots $, where
$S_k=\mathrm{Sym}^k(x_1,x_2,...,x_n)$.
We can further decompose each
 $V_\lambda \otimes S_i$ into irreducible modules of $\gl_n$;
once we do so, we find that $M(\lambda)$ has a simple $\gl_n$
spectrum. Note that $V_{\mu} \otimes S_1$ can be decomposed as
$V_{\mu-e_{11}^*} \oplus V_{\mu-e_{22}^*} \oplus \cdots \oplus
V_{\mu-e_{nn}^*}$ (taking
$V_{\mu-e_{ii}^*}=\{0\}$ if $\mu-e_{ii}^*$ is not dominant). We can
thus associate each $V_\mu$ for $\mu=\lambda-a_1 e_{11}^* - \cdots -
a_n e_{nn}^*$ in the decomposition of $M(\lambda)$ with a lattice
point $P_{\mu}=(-a_1,-a_2,\ldots,-a_n)\in \mathbb{Z}^n$. We draw a
directed edge from $P_{\mu}$ to $P_{\mu'}$ if $V_{\mu'}$ is in the
decomposition of  $V_{\mu} \otimes S_1$, and we say $P_{\mu'}$ is
\emph{smaller} than $P_{\mu}$. A key property of this graph is that any
$H_\df(\gl_n)$-submodule of $M(\lambda)$ intersecting the module
$V_{\mu}$ must necessarily contain $V_{\mu}$ and all $V_{\mu'}$ such
that $P_{\mu'}$ is reachable from $P_{\mu}$ by a walk along directed
edges. Recall that $L(\lambda)=M(\lambda)/\overline{M}(\lambda)$, where
$\overline{M}(\lambda)$ is the maximal proper $H_\df(\gl_n)$-submodule of
$M(\lambda)$. The aforementioned property guarantees that as a
$\gl_n$ module, $\overline{M}(\lambda)=\bigoplus_{s\in S} V_s$ for some set $S$
of vertices closed under walks, so that $L(\lambda)$ is finite dimensional
if and only if $\bar{S}$ (the complement of
$S$) is a finite set.

We now prove part (a). 
First, suppose that $L(\lambda)$ is finite dimensional.
The finiteness of $\bar{S}$ implies the existence of some $l$ such that
$(0,\ldots,0,-l-1)\in S$ (note that $(0,\ldots,0)\notin S$). Let
$\nu_n$ be the minimal such $l$. We define $S'$ as the set of
vertices that can be reached by walking from
$(0,\ldots,0,-\nu_n-1)$.
Because $S' \subseteq S$,
the Verma module $M(\lambda)$ must possess a submodule $M(\lambda-(0, \ldots, 0, \nu_n+1))$.
By considering the action of the Casimir element on $M(\lambda)$ and $M(\lambda-(0, \ldots, 0, \nu_n+1))$,
we get $P(\lambda)=P(\lambda-(0,\ldots,0, \nu_n+1))$.

Next, suppose that there exists $\nu_n \in \mathbb{N}_0$ such that
$P(\lambda)=P(\lambda-(0, \ldots, 0, \nu_n+1))$. The determinant formula of Theorem \ref{theorem:shapovalov} 
implies that the Verma module $M(\lambda)$ contains the submodule $M(\lambda-(0, \ldots, 0, \mu))$ for some 
$\mu \leq \nu_n$.
Define $S'$ to be the set of
vertices that can be reached by walking from
$(0,\ldots,0,-\mu)$.
Its
complement $\bar{S'}$ is finite, since for any vertex 
$(-a_1,\ldots,-a_n)$ of our graph, we have $\lambda_1-a_1\geq
\lambda_2-a_2\geq \cdots\geq \lambda_n-a_n$.
Because $\bar{S} \subseteq \bar{S'}$, $\bar{S}$ is finite, finishing the proof of (a). 
We note that explicitly, $\bar{S'}=\{(-a_1,\ldots,-a_n)|0\leq a_i\leq\lambda_i-\lambda_{i+1}, 0\leq a_n\leq \nu_n\}$ and the corresponding finite dimensional quotient
is $L'(\lambda)=M(\lambda)/(\sum_{1\leq i\leq n-1}{ H_\df(\gl_n) e_{i+1,i}^{\lambda_i-\lambda_{i+1}+1}v_\lambda+ H_\df(\gl_n)x_n^{\nu_n+1}v_\lambda})$. 

 Part (b) requires an additional argument. Namely, if $L(\lambda)$ is finite
dimensional, then it can also be considered as a lowest weight
representation. Let $\bar{b}=(b_1,\ldots,b_n)\in \bar{S}$ be the
vertex corresponding to the lowest weight of $L(\lambda)$.
 If the statement of (b) was wrong, there would be a vertex $\bar{e}=(e_1,\ldots,e_n)\in S$ with two nonzero coordinates, such that
$(e_1,\ldots,e_{i-1},e_i+1,e_{i+1},\ldots,e_n)\in \bar{S}$ for any $i$. Without loss
of generality, suppose $e_1, e_2\ne 0$.
 As we can walk along reverse
edges from $\bar{b}$ to both points $(e_1+1,e_2,\ldots,e_n)$ and
$(e_1,e_2+1,e_3,\ldots,e_n)$, we can also walk along reverse edges
to $\bar{e}$, which is a contradiction. This proves part (b) and
explains our terminology ``\emph{rectangular form}''.
\end{proof}

\begin{figure}[h]
    \centering
        \includegraphics[width=0.58\textwidth]{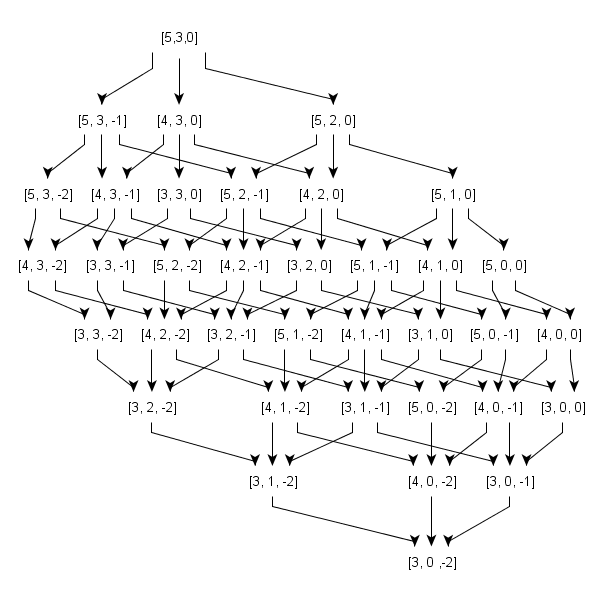}
    \caption{
    We use a graph to represent the rectangular prism corresponding to a finite dimensional representation $L((5,3,0))$ of $H_\df(\gl_3)$,
    with the highest weight of each $\gl_3$ module indicated.
    }
\end{figure}

The decomposition of $L(\lambda)$ as a $\gl_n$ module
provides the character formula for $L(\lambda)$ as the sum of the characters
of $\gl_n$ modules:
\begin{equation}
\tag{*} \chi_{\lambda;\df}= \sum_{0 \leq \lambda-\lambda' \leq \nu}
\frac{\sum_{w\in S_n}(-1)^w  e^{w (\lambda'+\rho)}}{\sum_{w\in
S_n}(-1)^w  e^{w\rho}}.
\end{equation}
As in the classical theory, this character allows us to calculate the decomposition of finite dimensional representations
into irreducible ones.

\begin{example}
Let us illustrate the decomposition of $L(\lambda)$ from the proof of Theorem \ref{theorem:classification};
for clarity, we will work with $\mathfrak{sl}_2$ representations instead of $\gl_2$ representations.
Using the notation of the proof,
$S_k=S^k(x_1,x_2)\cong V_{k}$, the irreducible 
$\mathfrak{sl}_2$ representation of dimension $k+1$.
By the Clebsch-Gordon formula,
\[
V_m \otimes V_{k}
\cong V_{m+k} \oplus V_{m+k-2} \oplus \cdots \oplus V_{m+k-2\min(k,m)}.
\]
We can use the above formula to draw the graph, as in Figure 2,
representing the decomposition of $L((2,0))$, with $\nu=(0,3)$, into
$\mathfrak{sl}_2$ modules. This representation is the quotient of
$M((2,0))/H_\df(\gl_2)e_{21}^3 v_{\lambda}$ by the submodules represented
by the shaded areas of the diagram, and $L((2,0))\cong
V_2 \oplus V_3 \oplus V_4 \oplus V_5$ as $\mathfrak{sl}_2$ modules.

\begin{figure}
    \centering
        \includegraphics[width=0.9\textwidth]{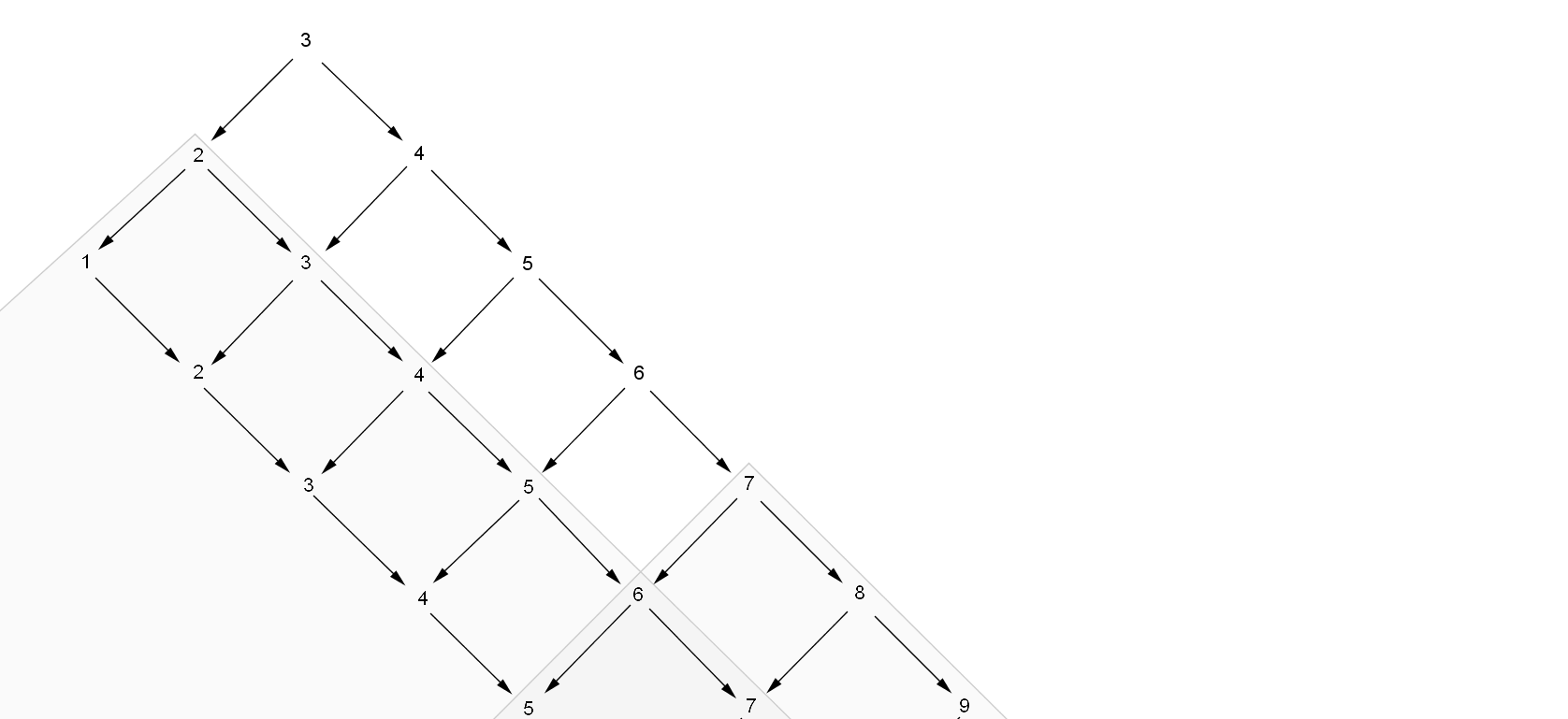}
        \caption{The decomposition of $L((2,0))$, with $\nu=(0,3)$, into $\mathfrak{sl}_2$ modules.}
\end{figure}
\end{example}

\begin{example}
For $H_\df(\gl_1)$,
the irreducible finite dimensional representation $L(\lambda)$, for $\lambda \in \mathbb{C}$, has character
$\chi_{\lambda,\df}= \sum_{\nu'=0}^{\nu} e^{\lambda-\nu'}$,
where $\nu$ is some nonnegative integer.
If we describe $H_\df(\gl_1)$ as in Example \ref{example:gl1}, we can
easily calculate the Casimir element to be $fe+g(h)$, where $g$ satisfies the equation $g(x)-g(x-1)=\phi(x)$.
Then, $\nu$ is the smallest nonnegative integer such that $g(\lambda)-g(\lambda-\nu-1)=0$.
\end{example}

\begin{example}
For $H_\df(\gl_2)$, the irreducible finite dimensional representations are necessarily of the form $L(\lambda)$ with
$\lambda=(\lambda_2+m,\lambda_2)$, where $\lambda_2 \in \mathbb{C},\  m\in \mathbb{N}_0$.
The character of $L(\lambda)$ equals
\[
\chi_{\lambda;\df}= \sum_{(0,0) \leq (\nu_1',\nu_2') \leq
(\nu_1,\nu_2)}
\frac{e^{(\lambda_2+m-\nu_1',\lambda_2-\nu_2')}-e^{(\lambda_2-\nu_2'-1,\lambda_2+m-\nu_1'+1)}}{1-e^{(-1,1)}}.
\]
Let
$f_1(\lambda,\mu)=P(\lambda_2+m+\frac{1}{2},\lambda_2-\frac{1}{2})-P(\lambda_2+m+\frac{1}{2}-\mu,\lambda_2-\frac{1}{2})$
and 
$f_2(\lambda,\mu)=P(\lambda_2+m+\frac{1}{2},\lambda_2-\frac{1}{2})-P(\lambda_2+m+\frac{1}{2},\lambda_2-\mu-\frac{1}{2})$.
 Again, $\nu_2$ is defined as the minimal nonnegative integer
satisfying
$f_2(\lambda,\nu_2+1)=0$,
while $\nu_1$ is either $m$ or the minimal nonnegative integer
satisfying
$f_1(\lambda,\nu_1+1)=0$.
 For instance, if $\df=\df_0 r_0$ with $\df_0 \neq 0$, then 
$f_2(\lambda,\mu)$ is a multiple of $\mu$, and so the only
solution to the equation $f_2(\lambda,\nu_2+1)=0$ is $\nu_2=-1$,
which is negative. Thus, $H_{\df_0 r_0}(\gl_2)$ has no finite
dimensional irreducible representations.
 If $\df=\df_0 r_0+\df_1 r_1$ with $\df_1 \neq 0$,
$P(\lambda)=\df_0(\lambda_1+\lambda_2)+\df_1((\lambda_1+\frac{1}{2})^2+(\lambda_1+\frac{1}{2})(\lambda_2-\frac{1}{2})+(\lambda_2-\frac{1}{2})^2)$,
so
$f_2(\lambda,\mu)=\df_1\mu\left(\frac{\df_0}{\df_1}+\lambda_1+2\lambda_2-\mu
\right)$. Thus, $L(\lambda)$ is finite dimensional if and only if
$\frac{\df_0}{\df_1}+\lambda_1+2\lambda_2$  is a positive integer. This agrees with the description of finite dimensional representations of $\mathfrak{sl}_3$.
\end{example}

\subsection{Existence of $L(\lambda)$ with a given shape}
\label{subsection:finiteexistence}
\begin{theorem}
For any $\gl_n$ dominant weight $\lambda$ and $\nu \in \mathbb{N}_0^n$ such
that $\nu_i\leq \lambda_i-\lambda_{i+1}$ for all $1\leq i\leq n-1$,
there exists a deformation $\df$, such that the irreducible
representation $L(\lambda)$ of $H_\df(\gl_n)$ is finite dimensional
and its character is given by \normalfont{(*)}.
\end{theorem}
\begin{proof}
Let $\lambda'=\lambda+\rho$. We can write
$\lambda'_i=\lambda'_n+k_i$ for $k_1>k_2>k_3>\cdots>k_{n-1}>k_n=0$
(we have strict inequalities because of the shift by $\rho$). Recall
that $P(\lambda)=\sum w_m H_{m}(\lambda')$ for $w_i$ defined as in
Theorem \ref{theorem:casimirgln}. Let
$\mu_i=(0,\ldots,\nu_i+1,0,\ldots,0)$. We will find $w_i$ such that
$P(\lambda')-P(\lambda'-\mu_i)=0$, while for all $0<\mu_i'<\mu_i$,
$P(\lambda')-P(\lambda'-\mu_i') \neq 0$. This implies that
there are embeddings of $M(\lambda'-\mu_i)$ into $M(\lambda')$ with an irreducible quotient
$L(\lambda')=M(\lambda')/\sum_i M(\lambda'-\mu_i)$, due to Theorem \ref{theorem:classification}.

Define $P_{mj}=P(\lambda')-P(\lambda'-\mu)$ for
$\mu=(0,\ldots,m+1,0,\ldots,0)$ with the $m+1$ at the $j$-th location.
We must prove that there exist $w$ such that
$P_{\nu_11}=\cdots=P_{\nu_nn}=0$ and
$P_{\nu_1'1},\ldots,P_{\nu_n'n}\neq 0$ for all $0<\nu_i'<\nu_i$. We
can write $P_{mj}=\sum_{i> 0} w_i R_{mj}^i$, where
\[
R_{mj}^N=\sum_{i_1+\ldots+i_n=N} (\lambda'_n+k_1)^{i_1}\cdots (\lambda'_n+k_{j-1})^{i_{j-1}} ((\lambda'_n+k_j)^{i_j}-(\lambda'_n+k_j-m-1)^{i_j})(\lambda'_n+k_{j+1})^{i_{j+1}}\cdots (\lambda'_n+k_n)^{i_n}.
\]
Note that the condition $P_{kj}=0$ determines a hyperplane $\Pi_{kj}$ in the space $(w_0,w_1,\ldots)$
($\Pi_{kj}$ might in fact be the entire space, but the following argument would be unaffected).
Hence, the intersection $\bigcap \Pi_{\nu_j j}$ belongs to the union $\bigcup_{j,0<\nu_j'<\nu_j} \Pi_{\nu_j',j}$
if and only if it belongs to some $\Pi_{\nu_j',j}$.
Thus, it suffices to show that $\{P_{\nu_11},\ldots,P_{\nu_nn},P_{\nu_l'l} \}$
are linearly independent as functions of $w_i$ for all $1 \leq l \leq n$ and $0 <\nu_l'<\nu_l$.
This condition of linear independence is satisfied if
\[
\det
\left(
\begin{matrix}
    R_{\nu_1 1}^{1} & R_{\nu_1 1}^{2} & \cdots & R_{\nu_1 1}^{n+1}\\
    R_{\nu_2 2}^{1} & R_{\nu_2 2}^{2} & \cdots & R_{\nu_2 2}^{n+1}\\
    \vdots & \vdots & \ddots & \vdots
    \\
    R_{\nu_n n}^{1} & R_{\nu_n n}^{2} & \cdots & R_{\nu_n n}^{n+1}\\
    R_{\nu_l' l}^{1} & R_{\nu_l' l}^{2} & \cdots & R_{\nu_l' l}^{n+1}\\
\end{matrix}
\right)
\neq 0.
\]

Now we shall prove that using column transformations, we can reduce the
above matrix to its evaluation at $\lambda'_n=0$.
We proceed by induction on the column number.
The elements of the first column, $R_{mj}^1$, are of degree zero with respect to $\lambda'_n$, so $R_{mj}^1=R_{mj}^1(0)$.
Suppose that using column transformations, all columns before column $p$ are reduced to their constant terms.
Now, we note that
\begin{align*}
\frac{\partial R_{mj}^p(\lambda_n')}{\partial \lambda'_n}  &= \frac{\partial}{\partial \lambda'_n} \left(
\sum_{i_1+\ldots+i_n=p} (\lambda_n'+k_1)^{i_1}\cdots ((\lambda_n'+k_j)^{i_j}-(\lambda_n'+k_j-m-1)^{i_j})\cdots \lambda_n'^{i_n}
\right) \\
&=
\sum_{i_1+\ldots+i_n=p-1} (i_1+i_2+\ldots+i_n+n)(\lambda_n'+k_1)^{i_1}\cdots ((\lambda_n'+k_j)^{i_j}-(\lambda_n'+k_j-m-1)^{i_j})\cdots \lambda_n'^{i_n}
\\
&=
(p+n-1) R_{mj}^{p-1}(\lambda_n').
\end{align*}
Thus, we see that $R_{mj}^p-R_{mj}^p(0)$ is a linear combination of $R_{mj}^{p-i}(0)$, the entries of the other columns:
\[R_{mj}^p(\lambda_n')=\sum_i \frac{1}{i!}\lambda_n'^i \left. \frac{\partial^iR_{mj}^p}{\partial \lambda_n'^i}  \right|_{\lambda'_n=0}
=
\sum_i \frac{(p+n-1)\cdots(p+n-i)}{i!}\lambda_n'^i R_{mj}^{p-i}(0)
=
\sum_i {{p+n-1} \choose i} R_{mj}^{p-i}(0) \lambda_n'^i
.\]
By selecting pivots of ${{p+n-1} \choose i} \lambda_n'^i$, we can eliminate every term except $R_{mj}^p(0)$.
By repeating this step, we reduce the matrix to its evaluation at $\lambda'_n=0$:
\[
\det
\left(
\begin{matrix}
    R_{\nu_1 1}^{1}(\lambda_n') & R_{\nu_1 1}^{2}(\lambda_n') & \cdots & R_{\nu_1 1}^{n+1}(\lambda_n')\\
    R_{\nu_2 2}^{1}(\lambda_n') & R_{\nu_2 2}^{2}(\lambda_n') & \cdots & R_{\nu_2 2}^{n+1}(\lambda_n')\\
    \vdots & \vdots & \ddots & \vdots
    \\
    R_{\nu_n n}^{1}(\lambda_n') & R_{\nu_n n}^{2}(\lambda_n') & \cdots & R_{\nu_n n}^{n+1}(\lambda_n')\\
    R_{\nu_l' l}^{1}(\lambda_n') & R_{\nu_l' l}^{2}(\lambda_n') & \cdots & R_{\nu_l' l}^{n+1}(\lambda_n')\\
\end{matrix}
\right)
=
\det
\left(
\begin{matrix}
    R_{\nu_1 1}^{1}(0) & R_{\nu_1 1}^{2}(0) & \cdots & R_{\nu_1 1}^{n+1}(0)\\
    R_{\nu_2 2}^{1}(0) & R_{\nu_2 2}^{2}(0) & \cdots & R_{\nu_2 2}^{n+1}(0)\\
    \vdots & \vdots & \ddots & \vdots
    \\
    R_{\nu_n n}^{1}(0) & R_{\nu_n n}^{2}(0) & \cdots & R_{\nu_n n}^{n+1}(0)\\
    R_{\nu_l' l}^{1}(0) & R_{\nu_l' l}^{2}(0) & \cdots & R_{\nu_l' l}^{n+1}(0)\\
\end{matrix}
\right).
\]

Let us now rewrite $R_{mj}^N(0)$:
\begin{align*}
R_{mj}^N(0)&=
\sum_{i_1+\ldots+i_n=N} k_1^{i_1}\cdots k_{j-1}^{i_{j-1}} (k_j^{i_j}-(k_j-m-1)^{i_j})k_{j+1}^{i_{j+1}}\cdots k_n^{i_n}
=
\sum_{i=0}^{N-1} H_{N-i-1}' \left(k_j^{i+1}-(k_j-m-1)^{i+1}\right)\\
&=
\sum_{i=0}^{N-1} H_{N-i-1} \left(k_j^{i+1}-(k_j-m-1)^{i+1}-k_j(k_j^i-(k_j-m-1)^{i}) \right)
=\sum_{i=0}^{N-1} H_{N-i-1} \left((m+1)(k_j-m-1)^i\right)
\end{align*}
where $H_{N-i}=\sum_{i_1+\ldots+i_n=N-i} k_1^{i_1} \cdots k_n^{i_n}$ and $H_{N-i}'=\sum_{i_1+\ldots+\widehat{i_j}+\ldots+i_n=N-i} k_1^{i_1} \cdots \widehat{k_j^{i_j}} \cdots k_n^{i_n}$.
The third equality is because $H_{N-i}'=H_{N-i}-k_j H_{N-i-1}$.
It is easy to see that the above determinant can be reduced further to
\[
\det
\left(
\begin{matrix}
    \nu_1+1 & (\nu_1+1)(k_1-\nu_1-1) & \cdots & (\nu_1+1)(k_1-\nu_1-1)^{n}\\
    \nu_2+1 & (\nu_2+1)(k_2-\nu_2-1) & \cdots & (\nu_2+1)(k_2-\nu_2-1)^{n}\\
    \vdots & \vdots & \ddots & \cdots
    \\
    \nu_n+1 & (\nu_n+1)(k_n-\nu_n-1) & \cdots & (\nu_n+1)(k_n-\nu_n-1)^{n}\\
    \nu_l'+1 & (\nu_l'+1)(k_l-\nu_l'-1) & \cdots & (\nu_l'+1)(k_l-\nu_l'-1)^{n}\\
\end{matrix}
\right)
=T\cdot\det
\left(
\begin{matrix}
    1 & k_1-\nu_1-1 & \cdots & (k_1-\nu_1-1)^{n}\\
    1 & k_2-\nu_2-1 & \cdots & (k_2-\nu_2-1)^{n}\\
    \vdots & \vdots & \ddots & \cdots
    \\
    1 & k_n-\nu_n-1 & \cdots & (k_n-\nu_n-1)^{n}\\
    1 & k_l-\nu_l'-1 & \cdots & (k_l-\nu_l'-1)^{n}\\
\end{matrix}
\right),
\]
where $T=(\nu_1+1)(\nu_2+1) \cdots (\nu_n+1) (\nu_l'+1)$ and the determinant is 
$\prod_{i=1}^n (k_l-k_i+\nu_i-\nu_l')\prod_{1\leq i < j \leq n}(k_j-k_i+\nu_i-\nu_j)$ by the Vandermonde determinant formula.
 Now, recalling the conditions $0\leq \nu_i\leq \lambda_i-\lambda_{i+1}=k_i-k_{i+1}-1$ we get $k_j-k_i+\nu_i-\nu_j<0$ for any $i<j$ and so
$\prod_{1\leq i < j \leq n}(k_j-k_i+\nu_i-\nu_j)$ is nonzero.
Similarly, we get $\prod_{i=1}^n (k_l-k_i+\nu_i-\nu_l')\ne 0$.
Hence, the determinant is nonzero, and so $\{P_{\nu_1,1},\ldots,
P_{\nu_n,n}, P_{\nu_l',l}\}$ are linearly independent as desired.
\end{proof}

$\ $

\section{Poisson Infinitesimal Cherednik Algebras}
\label{section:poisson}
Now we will study infinitesimal Cherednik algebras by using their Poisson analogues. 
The Poisson infinitesimal Cherednik algebras
are as natural as $H_\df(\gl_n)$,
and their theory goes along the same lines with some simplifications.
Although these algebras have not been defined before in the literature,
the authors of \cite{EGG} were aware of them, 
and technical calculations with these algebras are similar to those made in \cite{T}.
This approach provides another proof of Theorem \ref{theorem:firstcentralelementtheorem}.

Let $\dfe$ be a deformation parameter, $\dfe: V\times V^* \to S(\gl_n)$.
The Poisson infinitesimal Cherednik algebra $H_\dfe'(\gl_n)$ 
is defined to be the algebra
$S\gl_n \otimes S(V \oplus V^*)$
with a bracket defined on the generators by:
\begin{align*}
\{a,b\} &=[a,b] \text{ for } a,b \in \gl_n,\\
\{g,v\} &=g(v)  \text{ for } g\in \gl_n, v\in V\oplus V^*,\\
\{y,y'\}&=\{x,x'\} = 0 \text{ for } y,y' \in V, x,x' \in V^*,\\
\{y,x\} &=\dfe(y,x) \text{ for } y\in V, x \in V^*.
\end{align*}
This bracket extends to a Poisson bracket on $H_\dfe'(\gl_n)$  if and only if the Jacobi identity $\{\{x,y\},z\}+\{\{y,z\},x\}+\{\{z,x\},y\}=0$ holds
for any $x,y,z \in \gl_n \ltimes (V \oplus V^*)$.
As can be verified by computations analogous to \cite{EGG}, Theorem 4.2, the Jacobi identiy holds iff
$\dfe=\sum_{j=0}^k \dfe_j \mathfrak{r}_j$ where $\dfe_j \in \mathbb{C}$ and
$\mathfrak{r}_j$ is the coefficient of $\tau^j$ in the expansion of
$
(x,(1-\tau A)^{-1} y)\det (1-\tau A)^{-1}
$.
Actually, we can consider the infinitesimal Cherednik algebras of $\gl_n$ as quantizations of $H_\dfe'(\gl_n)$.
\begin{remark}[Due to Pavel Etingof]
\label{remark:glnpoisson}
Note that 
\[
\{y_i,x_j\}=\sum \zeta_l \mathfrak{r}_l(y_i, x_j) = \sum \zeta_l\frac{\partial \Tr(S^{l+1}A)}{\partial e_{ji}};
\]
this follows from
\[
\frac{\partial}{\partial B}(\det(1-\tau A)^{-1}) =\frac{\Tr(\tau B(1-\tau A)^{-1})}{\det(1-\tau A)}
\]
when $B=y_i \otimes x_j$. 
In fact, if $\{y_i, x_j\}=F_{ji}(A)$, the Jacobi identity implies that
$F_{ij}(A)=\frac{\partial F}{\partial e_{ij}}$ for some $GL(n)$ invariant function $F$,
and that $\Lambda^2 D_A(F)=0$,
where $D_A$ is the matrix with $(D_A)_{_{ij}}=\frac{\partial}{\partial e_{ij}}$.
One can then show that the only $GL(n)$ invariant functions $F$ satisfying this partial differential equation are
linear combinations of $\Tr(S^l A)$.
\end{remark}

Our main goal is to compute explicitly the Poisson center of the algebra $H_\dfe'(\gl_n)$.
As before, we set $\mathcal{Q}_k$ to be
the coefficient of $(-t)^k$ in the expansion of $\det(1-tA)$,
$\pt_k=\sum_{i=1}^n x_i \{\mathcal{Q}_k,y_i\}$,
and $\dfe(z)=\dfe_0+\dfe_1 z + \dfe_2 z^2 + \cdots$.
\begin{theorem}
\label{theorem:poissoncentergln}
The Poisson center $\zpois(H_\dfe'(\gl_n))=\mathbb{C}[\pt_1+c_1,\pt_2+c_2,\ldots,\pt_n+c_n]$,
where $(-1)^i c_i$ is the coefficient of $t^i$ in the series
\[
c(t)=\Res_{z=0} \dfe(z^{-1}) \frac{\det(1-tA)}{\det(1-zA)} \frac{1}{1-t^{-1} z} \frac{dz}{z}.
\]
\end{theorem}
\begin{proof}
First, we claim that $\zpois(H_0'(\gl_n))=\mathbb{C}[\pt_1,\ldots,\pt_n]$.
The inclusion 
$\mathbb{C}[\pt_1,\ldots,\pt_n] \subseteq \zpois(H_0'(\gl_n))$ 
is straightforward,
while the reverse inclusion follows from 
the structure of the coadjoint action of the Lie group corresponding to 
$\gl_n \ltimes (V \oplus V^*)$ (as in the proof of \cite{T}, Theorem 2).

We prove that the Poisson center of $H_0'(\gl_n)$
can be lifted to the Poisson center of $H_\dfe'(\gl_n)$
by verifying that $\pt_i+c_i$ are indeed Poisson central.
Since $\pt_k \in \zpois(H_0(\gl_n))$ and $c_k \in \zpois(S(\mathfrak{gl_n}))$, 
$\pt_k+c_k$ 
Poisson-commutes with elements of $S(\gl_n)$.
We can define an anti-involution on $H_\dfe'(\gl_n)$ 
that acts on basis elements by taking $e_{ij}$ to $e_{ji}$ and
$y_i$ to $x_i$.
By using the arguments explained in the proof of Theorem 2 in \cite{T}, we can show that $\pt_k$
is fixed by this anti-involution, 
while $c_k$ is also fixed since it lies in $\zpois(S(\mathfrak{gl_n}))$.
Applying this anti-involution,
we see that 
it suffices to show that $c_k$ satisfies $\{\pt_k+c_k,y_l\}=0$ for basis elements $y_l\in V$.


First, notice that if $g \in S(\gl_n)$, 
then $\{g,y_l\}=\sum_{i,j=1}^{n} \frac{\partial g}{\partial e_{ij}}\{e_{ij},y_l\}$,
and together with the equation
$\{\{\mathcal{Q}_k,y_i\}, y_l\}=0$ 
(see the proof of Lemma 2.1 in \cite{T}),
we get
\[
\{\pt_k,y_l\}=\left\{\sum_{i=1}^n x_i \{\mathcal{Q}_k,y_i\}, y_l \right\}=\sum_{i=1}^n \{x_i,y_l\}\{\mathcal{Q}_k,y_i\}=-\sum_{i=1}^n\left(\Res_{z=0} \dfe(z^{-1}) \frac{\tr (x_i(1-zA)^{-1}y_l)}{z\det(1-zA)} dz \right)\{ \mathcal{Q}_k, y_i\}.
\]
Thus, we have
\[
\{\pt_k+c_k,y_l\}=\sum_{i,j=1}^{n} \frac{\partial c_k}{\partial e_{ij}}\{e_{ij},y_l\}-\sum_{i=1}^n\left(\Res_{z=0} \dfe(z^{-1}) \frac{\tr (x_i(1-zA)^{-1}y_l)}{z\det(1-zA)} dz \right)\{ \mathcal{Q}_k, y_i\}.
\]
Hence, $\{\pt_k+c_k,y_l\}=0$ is equivalent to the system of partial differential equations:
\[
\sum_{i,j=1}^{n} \frac{\partial c_k}{\partial e_{ij}}\{e_{ij},y_l\}=\sum_{i=1}^n\left(\Res_{z=0} \dfe(z^{-1}) \frac{\tr (x_i(1-zA)^{-1}y_l)}{z\det(1-zA)} dz \right)\{ \mathcal{Q}_k, y_i\}.
\]
Multiplying both sides by $(-t)^k$ and summing over $k=1,\ldots,n,$ we obtain an equivalent single equation
\[
\sum_{i,j=1}^{n} \frac{\partial c(t)}{\partial e_{ij}}\{e_{ij},y_l\}=\sum_{i=1}^n\left(\Res_{z=0} \dfe(z^{-1}) \frac{\tr (x_i(1-zA)^{-1}y_l)}{z\det(1-zA)} dz \right)\{\det(1-tA), y_i\}.
\]

Since all terms above are $GL(n)$ invariant and diagonalizable matrices are dense in $\gl_n$,
we can set $A=\mathrm{diag}(a_1, \ldots, a_n)$:
\begin{align*}
\frac{\partial c(t)}{\partial a_l} y_l&=\left(\Res_{z=0}
\frac{\dfe(z^{-1})}{z(1-z a_l)\det(1-zA)}dz \right) \{\det(1-tA), y_l\}\\
&=
\left(\Res_{z=0}
\frac{\dfe(z^{-1})}{z(1-z a_l)\det(1-zA)}dz \right)
 \frac{\partial \det(1-tA)}{\partial a_l} y_l
 \\
&=
-\left(\Res_{z=0}
\frac{\dfe(z^{-1})}{z(1-z a_l)\det(1-zA)}dz \right) 
\frac{t\det(1-tA)}{1-t a_l} y_l,
\end{align*}
and it is easy to see that $c(t)$ satisfies the above equation.
\end{proof}
\begin{example}
\label{example:l1}
In particular, $c_1=\sum_{i=0}^k \dfe_i \Tr S^{i+1} A$.
\end{example}

\begin{remark}
Another way of writing the formula for $c_k$ is
\[
c_k= \Res_{z=0} \dfe(z^{-1}) G_k(z) \frac{dz}{z^2},
\]
where $G_k(z)=\sum z^m y_{m,k}(A)$ and $y_{m,k}(A)=\chi\underbrace{(m,1,\ldots,1)}_{k}$,
the character of an irreducible $\gl_n$ module corresponding to a hook Young diagram\footnote{This formula follows from the fact that in the Grothendieck ring of finite dimensional $\gl_n$ representations, $[\bigwedge^k V \otimes S^m V]-[\bigwedge^{k+1} V \otimes S^{m-1} V]+\cdots+(-1)^m[\bigwedge^{k+m} V]=[V_{(m+1,1,\ldots,1)}]$ due to Pieri's formula.}.
This provides a better insight for the quantization construction.
\end{remark}
\begin{remark}
We expect that for any $a_1, \ldots, a_n \in \mathbb{C}$, the induced symplectic structure on $\mathrm{Spec}(S(\gl_n) \otimes S(V \oplus V^*)/(\tau_1+c_1-a_1, \ldots, \tau_n+c_n-a_n))$ has only finitely many symplectic leaves.
\end{remark}

$\ $

\section{Passing from Commutative to Noncommutative Algebras}

Note that $\{g, y\} \in S(\gl_n) \otimes V$ for $g \in S(\gl_n)$ and $y \in V$;
we can thus identify $\{g,y\}=\sum_{i=1}^n h_i \otimes y_i \in H_\dfe'(\gl_n)$
with the element $\sum_{i=1}^n \mathrm{Sym}(h_i) y_i \in H_\dfe(\gl_n)$.

\begin{lemma}
\label{lemma:com2noncom}
\[
[\Tr S^{k+1} A, y] = \left\{\sum_{j=0}^{k}\frac{(-1)^j}{k+n+1} {k+n+1 \choose j+1}\Tr S^{k+1-j} A, y \right\}.
\]
\end{lemma}
\begin{proof}
It is enough to consider the case $y=y_1$.
Recall that
$\Tr S^{k+1}(A)$ can be written as a sum of degree $k+1$ monomials of form $e_{1,i_1} \cdots e_{1,i_{s_1}} e_{2, i_{s_1+1}} \cdots e_{2, i_{s_1+s_2}} \cdots e_{n,i_{s_1+\cdots+s_n}}$ 
where $s_1+\cdots + s_n=k+1$ and the sequence $\{i_k\}$ is a permutation of the sequence of $s_1$ ones, $s_2$ twos, and so forth;
for conciseness, we will denote the above monomial by $e_{1,i_1} \cdots e_{n,i_{k+1}}$.
The only terms of $\Tr S^{k+1} A$ that contribute to $[\Tr S^{k+1} A,y_1]$ and to $\{\Tr S^{k+1} A, y_1\}$ 
have $s_1 \geq 1$.
Since to compute $[\Tr S^{k+1} A,y_1]$ we first symmetrize $\Tr S^{k+1}A$,
we will compute $[\mathrm{Sym}(e_{1,i_1} \cdots e_{n,i_{k+1}}),y_1]-\{\mathrm{Sym}(e_{1,i_1} \cdots e_{n,i_{k+1}}),y_1\}$.
For both the Lie bracket and the Poisson bracket, we use Leibniz's rule to compute the bracket,
but whereas in the Poisson case we can transfer the resulting elements of $V$ to the right since the Poisson algebra is commutative,
in the Lie case when we do so extra terms appear.

Consider a typical term that may appear after we use Leibniz's rule to compute $[\Tr S^{k+1} A,y_1]$:
\[
\cdots y_{j_0} \cdots e_{j_1 j_0} \cdots e_{j_2 j_1} \cdots e_{j_N j_{N-1}} \cdots
\] 
When we move $y_{j_0}$ to the right, we get, besides
$\cdots e_{j_1 j_0} \cdots e_{j_2 j_1} \cdots e_{j_N j_{N-1}} \cdots y_{j_0}$,
additional residual terms like 
$
-\cdots e_{j_2 j_1} \cdots e_{j_N j_{N-1}} \cdots y_{j_1}
$
and
$
\cdots e_{j_3 j_2} \cdots e_{j_N j_{N-1}} \cdots y_{j_2}
$,
up to
$
(-1)^N \cdots y_{j_N}
$.
Without loss of generality, we can consider only the last expression,
since the others will appear in the smaller chains 
\[
\cdots y_{j_0} \cdots e_{j_1 j_0} \cdots \widehat{e_{j_2j_1}} \cdots \widehat{e_{j_3j_2}}\cdots \widehat{e_{j_N j_{N-1}}}
\] 
and
\[
\cdots y_{j_0} \cdots e_{j_1 j_0} \cdots e_{j_2 j_1} \cdots \widehat{e_{j_3j_2}}\cdots \widehat{e_{j_N j_{N-1}}},
\] 
and so forth, with the same coefficients.
For notational convenience, we let $z_1$ denote the coefficient of $y_{j_N}$ in the residual term,
i.e., the term represented by the ellipsis: $(-1)^N \underbrace{\cdots}_{z_1} y_{j_N}$.
Then, $z_1 y_{j_N}$ is a term in the expression
$(-1)^N \{ z_1 e_{j_N 1}, y_{1}\}$,
which appears in $(-1)^N \{\Tr S^{k+1-N}A,y_1\}$.
Thus, we can write
\[
[\Tr S^{k+1} A, y_1]=\left\{\sum_{N=0}^{k} (-1)^N C_N \Tr S^{k+1-N}A,y_1 \right\}
\]
for some coefficients $C_N$.

Next, we compute $C_N$.
We first count how many times $z_1 y_{j_N}$ appears in $\{\Tr S^{k+1-N}A,y_1\}$.
Notice that since $z_1$ is the product of $k-N$ $e_{jl}$'s, 
we can insert $e_{j_N 1}$ in $k-N+1$ places to obtain $z_2$ such that $\{z_2,y_1\}$ contains $z_1 y_{j_N}$.

Now we compute the coefficient of $z_2$ in $\Tr S^{k+1-N} A$. 
As noted before,
$\Tr S^{k+1-N}(A)$ can be written as a sum of degree $k+1-N$ monomials of form $e_{1,i_1} \cdots e_{1,i_{s_1}} e_{2, i_{s_1+1}} \cdots e_{2, i_{s_1+s_2}} \cdots e_{n,i_{k+1-N}}$.
Any term that is a permutation of those $k+1-N$ unit matrices 
will appear in the symmetrization of $\Tr S^{k+1-N} A$.
We count the number of sequences $i_1, \ldots, i_{k+1-N}$
such that $z_2$ is the product
of the elements
$e_{1,i_1}, \ldots, e_{n,i_{k+1-N}}$ (in some order);
this tells us the multiplicity of $z_2$ in the symmetrization of $\Tr S^{k+1-N} A$.
Suppose
$z_2=e_{1,i_1} \cdots e_{n,i_{k+1-N}}$ for a certain sequence $i_1, \ldots, i_{k+1-N}$. 
Then,
$z_2=e_{1,i_1'} \cdots e_{n,i_{k+1-N}'}$ 
if and only if
$i_{s_1+\cdots+s_{j-1}+1}', \ldots, i_{s_1+\cdots+s_{j}}'$
is a permutation of
$i_{s_1+\cdots+s_{j-1}+1}, \ldots, i_{s_1+\cdots+s_j}$
for all $j$.
Thus, $z_2$ appears $s_1! s_2! \cdots s_n!$
times in $\Tr S^{k+1-N} A$.
Since each term has coefficient $\frac{1}{(k-N+1)!}$ in the symmetrization,
$z_2$ appears with coefficient 
\[
\frac{s_1! s_2! \cdots s_n!}{(k-N+1)!}
\]
in the symmetrization of $\Tr S^{k+1-N} A$.
In conjunction with the previous paragraph, we see that
$z_1 y_{j_N}$ appears
\[
\frac{s_1! s_2! \cdots s_n!}{(k-N+1)!} \times (k-N+1) = \frac{s_1! s_2! \cdots s_n!}{(k-N)!}
\]
times in $\{\Tr S^{k+1-N}A,y_1\}$.

It remains to calculate how many times $z_1 y_{j_N}$ appears in $[\Tr S^{k+1} A, y_1]$.
Recall that $z_1$ is obtained from a term like:
\[
\cdots e_{j_0 1} \cdots e_{j_1 j_0} \cdots e_{j_2 j_1} \cdots e_{j_N j_{N-1}} \cdots
\]
where the ordered union of the ellipsis equals $z_1$.
Thus, $z_1$ comes from terms of the following form: we choose arbitrary numbers $j_0, \ldots, j_{N-1}$,
and insert $e_{j_0 1}, e_{j_1 j_0},\ldots, e_{j_N j_{N-1}}$ into $z_1$.
There are 
\[
\frac{(k+1)(k)\cdots(k+1-N)}{(N+1)!}
\]
ways for this choice for any fixed $j_0, \ldots, j_{N-1}$.
Any such term $z_3$ appears in $\Tr S^{k+1} A$ with coefficient
\[
\frac{s_1'! \cdots s_n'!}{(k+1)!}
\]
where $s_l'$ is the total number of $e_{li}$'s (for some $i$) in $z_3$,
i.e., $s_l+ \text{number of } j_i\text{'s with }j_i=l$, $0 \leq i <N$.

Combining the results of the last two paragraphs, we see that
$\{\Tr S^{k+1-N}A,y_1\}$ must appear with coefficient
\[
\left (\frac{(k+1)(k)\cdots(k+1-N)}{(N+1)!} \sum \frac{s_1'! \cdots s_n'!}{(k+1)!} \right)\left.\middle/ \frac{s_1! s_2! \cdots s_n!}{(k-N)!}\right.
=
\frac{1}{(N+1)!} \sum \frac{s_1'! \cdots s_n'!}{s_1! s_2! \cdots s_n!}
,
\]
where the summation is over all length-$N$ sequences $\{j_l\}$ of integers from 1 to $n$.
We claim that 
\[
\frac{\sum s_1'! \cdots s_n'!}{ s_1! \cdots s_n!}=(k+n) \cdots (k+n-N+1).
\]
To see this, notice that $\frac{\sum s_1'! \cdots s_n'!}{ s_1! \cdots s_n!}$
is the coefficient of $t^N$ in the expression
\[
N! \prod_{i=1}^{n} \left(1+(s_i+1) t + \frac{(s_i+1)(s_i+2)}{2!} t^2 + \cdots \right).
\]
The above generating function equals $N! \prod_{i=1}^{n} (1-t)^{-(s_i+1)}=N! (1-t)^{-(k+1-N+n)}$,
and the coefficient of $t^N$ in this expression is $(k+n) \cdots (k+n-N+1)$.

Finally, we arrive at the simplified coefficient of $\{\Tr S^{k+1-N}A,y_1\}$:
\[
C_N=\frac{1}{(N+1)!} \sum \frac{s_1'! \cdots s_n'!}{s_1! s_2! \cdots s_n!}=
\frac{(k+n) \cdots (k+n-N+1)}{(N+1)!},
\]
as desired.
\end{proof}

Now we will give an alternative proof of Theorem \ref{theorem:firstcentralelementtheorem}.
\begin{proof}
Let $f(z)$ be the polynomial satisfying $f(z)-f(z-1)=\partial^n(z^n \dfe(z))$ and $g(z)=z^{1-n}\frac{1}{\partial^{n-1}} f(z)$
(in the expression for $g(z)$, we discard any negative powers of $z$).
Note that if $g(z)=g_{k+1}z^{k+1}+\cdots+g_1 z$, then
	\[
	\dfe(z)=\sum_{j=1}^{k+1}\sum_{i=0}^{j-1} \frac{1}{j+n}{j+n \choose i+1} (-1)^i g_j z^{j-1-i},
	\ \ \
\dfe_{j-1}=\sum_{i=0}^{k-j+1}\frac{1}{j+i+n}{j+n+i \choose i+1}(-1)^i g_{j+i}.
\]
Lemma \ref{lemma:com2noncom} allows us to write
\begin{align*}
\left[\sum_{j=1}^{k+1} g_{j} \tr S^{j}A,y\right]&=\left\{\sum_{j=1}^{k+1} \sum_{i=0}^{j-1} \frac{1}{j+n}{j+n \choose i+1} (-1)^i g_j\tr S^{j-i} A, y\right\}\\
&=
\left\{ \sum_{j=1}^{k+1} \sum_{i=0}^{k-j+1} \frac{1}{j+i+n}{j+i+n \choose i+1} (-1)^i g_{j+i}\tr S^{j} A, y\right\}
=
\left\{\sum_{j=1}^{k+1} \dfe_{j-1}\tr S^{j}A, y\right\}.
\end{align*}
Hence,
\begin{align*}
[t_1,y]=\sum_{i=1}^n [x_i,y] y_i=\sum_{i=1}^n \{x_i,y\} y_i=-\left\{\sum_{j=1}^{k+1} \dfe_{j-1}\tr S^{j}A, y\right\}=-\left[\sum_{j=1}^{k+1} g_{j} \tr S^{j}A,y\right],
\end{align*}
where the third equality follows from the fact that
$\pt_1+\sum_{j=1}^{k+1} \dfe_{j-1}\tr S^{j}A$ is Poisson-central in $H_\dfe'(\gl_n)$ (see Example \ref{example:l1}).
Thus, we get $t_1'=t_1+C'$, where
\[
C'=\sum_{j=1}^{k+1} g_{j} \tr S^{j}A=\Res_{z=0} g(z^{-1})\det(1-zA)^{-1} z^{-1} dz.
\]
\end{proof}
\begin{remark}
Comparing the formula for $c_1$ in Example \ref{example:l1}
to the one from Theorem \ref{theorem:firstcentralelementtheorem},
we see that they differ only by a change $z \df(z) \rightsquigarrow g(z)$.
We expect that a similar \emph{twist} of the formula for $c(t)$ 
given in Theorem \ref{theorem:poissoncentergln} will provide
the formulas for the actions of the generators of $\mathfrak{z}(H_\df(\gl_n))$ on the Verma module $M(\lambda-\rho)$.
\end{remark}


$\ $

\section{Algebras $H_\dfe(\spn_{2n})$ and $H_\dfe'(\spn_{2n})$}
\label{section:sp2n}

Let $V$ be the standard $2n$-dimensional representation of $\spn_{2n}$ with symplectic form $\omega$, and let $\dfe: V\times V \rightarrow U(\spn_{2n})$ be an $\spn_{2n}$ invariant bilinear form.
The \emph{infinitesimal Cherednik algebra} $H_\dfe(\spn_{2n})$ is defined as the quotient of $U(\spn_{2n})\ltimes T(V)$
 by the relation $[x,y]=\dfe(x,y)$ for all $x,y \in V$, such that $H_\dfe(\spn_{2n})$ satisfies the PBW property.
In \cite{EGG}, Theorem 4.2, it was shown that $H_\dfe(\spn_{2n})$ satisfies the PBW property if and only if
$\dfe=\sum_{j=0}^k \dfe_{2j} r_{2j}$
where $r_j$ is the symmetrization of the coefficient of $z^j$ in the expansion of
\begin{eqnarray*}
    \omega(x,(1-z^2 A^2)^{-1} y) \det(1-z A)^{-1}=r_0(x,y)+r_2(x,y)z^2 + \cdots.
\end{eqnarray*}
Note that for $A \in \spn_{2n}$, the expansion of $\det(1-z A)^{-1}$ contains only even powers of $z$.

\begin{remark}
For $\dfe_0 \neq 0$, there is an isomorphism
$H_{\dfe_0 r_0}(\spn_{2n}) \cong U(\spn_{2n})\ltimes A_n$, where $A_n$ is the $n$-th Weyl algebra (see \cite{EGG} Example 4.11).
Thus, we can regard $H_\dfe (\spn_{2n})$ as a deformation of $U(\spn_{2n})\ltimes A_n$.
\end{remark}

Choose a basis $v_j$ of $V$, so that
\[
\omega(x,y)=x^T J y,
\]
with
\[
J=
\left(
\begin{matrix}
0 & 1 & 0 & 0 & \cdots  &0 & 0\\
-1& 0 & 0 & 0 & \cdots & 0 & 0\\
0 & 0 & 0 & 1 & \cdots & 0 & 0\\
0 & 0 & -1& 0 & \ddots & \vdots & \vdots\\
\vdots & \vdots & \vdots & \ddots & \ddots & \vdots & \vdots\\
0 & 0 & 0  & 0 & \cdots & 0 & 1\\
0 & 0 & 0  & 0 & \cdots & -1 & 0 
\end{matrix}
\right).
\]
As before, we study the noncommutative infinitesimal Cherednik algebra $H_\dfe(\spn_{2n})$ by considering its Poisson analogue
$H_\dfe'(\spn_{2n})$.
We define
$\sum_{i=0}^{n} \mathcal{Q}_i z^{2i} = \det(1-zA)$ and
\[
		\st_i=(-1)^{i-1}\sum_{j=1}^{2n} \{\mathcal{Q}_{i},v_j\} v_j^*,
\] 
where $\{v_j^*\}$ is dual to $\{v_j\}$ (that is, $\omega(v_i, v_j^*)=\delta_{ij}$).
When viewed as an element of $\mathbb{C}[\spn_{2n} \ltimes V]$,
\[
\st_i=-\sum_{j=0}^{i-1} \mathcal{Q}_j\omega (A^{2i-1-2j} v, v),
\]
so $\st_i$ is $\spn_{2n}$ invariant and independent of the choice of basis $\{v_i\}$.


\begin{proposition}
The Poisson center of $H_0' (\spn_{2n})$ is $\mathbb{C}[\st_1,\ldots, \st_n]$.
\end{proposition}
\begin{proof}
We will follow a similar approach as in the proof of Theorem 2.1, \cite{T}.
Let $L$ be the Lie algebra $\mathfrak{sp}_{2n} \ltimes V$ and $S$ be the Lie group of $L$.
We need to verify that $\mathbb{C}[\st_1, \ldots, \st_n] = \zpois(H_0'(\spn_{2n}))$, 
the latter being identified with $\mathbb{C}[L^*]^S$.
Let $M\subset L$ be the $2n$-dimensional subspace consisting of elements of the form
\[
y=
\left\{
\left(
\begin{matrix}
0      & y_{12} & 0              & \cdots & 0    & 0\\
y_{21} & \ddots & \ddots         & \ddots & \vdots     & \vdots \\
0      & \ddots & 0              & y_{2n-3,2n-2} & 0 &0 \\
0      & 0      & y_{2n-2, 2n-3} & 0      & 0   & 0 \\
0      & \cdots & 0              & 0      & 0   & y_{2n-1,2n} \\
0      & \cdots & 0              & 0      & 0   & 0
\end{matrix}
\right), 
\left(
\begin{matrix}
0\\
0\\
\vdots \\
0 \\
0\\
y_{2n}
\end{matrix}
\right)
\right\},
\]
where all the $y$'s belong to $\mathbb{C}$.
In what follows, we identify $L^*$ and $L$ via the non-degenerate pairing,
so that the coadjoint action of $S$ is on $L$.
We use the following two facts proved in \cite{K}:
first, that the orbit of $M$ under the coadjoint action of $S$ on $L^*$ is dense in $L^*$;
and second, that $\mathbb{C}[L^*]^S \cong \mathbb{C}[f_1, \ldots, f_n]$,
where 
\[
\left. f_i \right|_M(y)=\sigma_{i-1}(y_{2,1} y_{1,2}, y_{3,2} y_{2,3}, \ldots, y_{2n-2,2n-3} y_{2n-3,2n-2}) y_{2n-1, 2n} y_{2n}^2
\]
and $\sigma_j$ is the $j$-th elementary symmetric polynomial.
It is straightforward to see that $\left. \st_i \right|_M =  f_i$, and so $\mathbb{C}[L^*]^S \cong \mathbb{C}[\st_1, \ldots, \st_n]$ as desired.
\end{proof}
As before, let $\dfe(z)=\dfe_0 + \dfe_2 z^2+\dfe_4 z^4 + \cdots$.
  
\begin{theorem}
\label{theorem:sp2npoissoncenter}
The Poisson center $\zpois(H_\dfe'(\spn_{2n}))=\mathbb{C}[\st_1+c_1,\st_2+c_2,\ldots,\st_n+c_n]$,
where $(-1)^{i-1}c_i$ is the coefficient of $t^{2i}$ in the series
\[
c(t)=2\Res_{z=0} \dfe(z^{-1}) \frac{\det(1-tA)}{\det(1-zA)} \frac{z^{-1}}{1-z^2 t^{-2}} dz.
\]
\end{theorem}
\begin{proof}
Since $c_i \in \zpois(S(\spn_{2n}))$, $\{\st_i+c_i,g\}=0$ for any $g\in S(\spn_{2n})$,
and so it suffices to show that  $\{\st_i+c_i,v\}=0$  for all $v\in V$.
By the Jacobi rule,
\[
\{\st_i,v\}= (-1)^{i-1} \sum_{j} \{\mathcal{Q}_i, v_j\} \{v_j^*,v\}+(-1)^{i-1} \sum_{j} \{\{\mathcal{Q}_i,v_j\},v\}v_j^*.
\]
Thus,
\begin{equation}
\label{eqn:sp2npoisson}
\{\st_i+c_i,v\}=(-1)^{i-1} \sum_{j} \{\mathcal{Q}_i, v_j\} \{v_j^*,v\}+(-1)^{i-1} \sum_{j} \{\{\mathcal{Q}_i,v_j\},v\}v_j^* + \{c_i,v\}.
\end{equation}
In the case of $H_\dfe'(\gl_n)$, $\sum_j \{ \{ \mathcal{Q}_i, y_j\}, y\} x_j=0$ by straightforward application of properties of the determinant.
However, for $H_\dfe'(\spn_{2n})$, $\sum_{j} \{\{\mathcal{Q}_i,v_j\},v\}v_j^* \neq 0$.
To calculate this sum, let $B$ be a basis of $\mathfrak{sp}_{2n}$ (the basis elements are given in the Appendix,
but for the purposes of this section, the specific elements are not needed).
Write
\[
\sum_{j} \{\{\mathcal{Q}_i,v_j\},v\}v_j^*=
\sum_{j} \left\{\sum_{e\in B} \frac{\partial \mathcal{Q}_i}{\partial e} e(v_j),v \right\}v_j^*
=
\sum_{j}\left( \sum_{e\in B} \frac{\partial \mathcal{Q}_i}{\partial e}  \{e(v_j),v\}v_j^*+ \left\{\frac{\partial \mathcal{Q}_i}{\partial e},v\right\} e(v_j)v_j^* \right).
\]
\begin{lemma}
\label{lemma:sp2nkey}
\[
\sum_{j}\sum_{e\in B}\left\{\frac{\partial \mathcal{Q}_i}{\partial e} ,v \right\} e(v_j)v_j^*=0.
\]
\end{lemma}
The proof of this lemma is quite technical and is provided in the Appendix.

Using the fact that $\sum_{j} \{\{\mathcal{Q}_i,v_j\},v\}v_j^*=\sum_{j} \sum_{e\in B} \frac{\partial \mathcal{Q}_i}{\partial e}  \{e(v_j),v\}v_j^*$,
we can restrict (\ref{eqn:sp2npoisson}) to diagonal matrices,
which are spanned by elements
$e_i=\mathrm{diag}(0,\ldots,1,-1,0,\ldots,0)$ with 1 at the $2i-1$-th coordinate. Thus,
the condition $\{\st_i+c_i,v\}=0$ is equivalent to:
\begin{align*}
0&=(-1)^{i-1} \sum_{j} \sum_k \frac{\partial \mathcal{Q}_i}{\partial e_k} \{e_k, v_j\} \{v_j^*,v\}+(-1)^{i-1}\sum_k \left(\frac{\partial \mathcal{Q}_i}{\partial e_k}\{ v_{2k-1},v\}v_{2k}+\frac{\partial \mathcal{Q}_i}{\partial e_k}\{ v_{2k},v\}v_{2k-1}\right)+ \sum_k \frac{\partial c_i}{\partial e_k}\{e_k, v\}\\
&=2(-1)^{i-1}\sum_k \frac{\partial \mathcal{Q}_i}{\partial e_k} (v_{2k-1}\{v_{2k},v\}+v_{2k}\{v_{2k-1},v\})
 + \sum_k \frac{\partial c_i}{\partial e_k}\{e_k, v\}.
\end{align*}
Multiplying the above equation by $(-1)^{i-1} t^{2i}$ and summing over $i$ for $i=1,\ldots,n$, 
the required condition transforms into:
\[
0=2\sum_k \frac{\partial \det(1-tA)}{\partial e_k} (v_{2k-1}\{v_{2k},v\}+v_{2k}\{v_{2k-1},v\})+
\sum_k \frac{\partial c(t)}{\partial e_k}\{e_k, v\}.
\]

It suffices to check this condition for basis vectors $v=v_{2s-1}$ and $v=v_{2s}$.
Substituting, we get
\[
0=2\sum_k \frac{\partial \det(1-tA)}{\partial e_k} (v_{2k-1}\{v_{2k},v_{2s-1}\}+v_{2k}\{v_{2k-1},v_{2s-1}\})+\frac{\partial c(t)}{\partial e_s}v_{2s-1}
\] 
and
\[
0=2\sum_k \frac{\partial \det(1-tA)}{\partial e_k} (v_{2k-1}\{v_{2k},v_{2s}\}+v_{2k}\{v_{2k-1},v_{2s}\})-\frac{\partial c(t)}{\partial e_s}v_{2s}.
\] 
These last two formulas both reduce to
\begin{align*}
\frac{\partial c(t)}{\partial e_s} &= -2\frac{\partial \det(1-tA)}{\partial e_s}\{v_{2s},v_{2s-1} \}\\
&=-2\frac{\partial \det(1-tA)}{\partial e_s}\left(\Res_{z=0} \dfe(z^{-1})\omega(v_{2s},(1-z^2 A^2)^{-1} v_{2s-1}) \det(1-z A)^{-1}z^{-1} dz\right)\\
&=
2\Res_{z=0} \dfe(z^{-1})\frac{\partial \det(1-tA)}{\partial e_s}\frac{1}{1-z^2 \lambda_s^2} \det(1-z A)^{-1}z^{-1} dz,
\end{align*}
and it is straightforward to verify that $c(t)$ satisfies the above equation.
\end{proof}

We now briefly consider the center of $H_\df(\spn_{2n})$.
Let $\beta_i\in U(\spn_{2n})$ be the symmetrization of $\mathcal{Q}_i$, and let
\[
		t_i=(-1)^{i-1}\sum_{j=1}^{2n} [\beta_{i},v_j] v_j^*.
\] 
Clearly, $t_i$ is independent of the choice of basis $\{ v_j\}$ and $\spn_{2n}$ invariant. 


\begin{conjecture}
\footnote{This conjecture was recently proved in \cite{LT}, using another presentation of $H_\dfe (\spn_{2n})$.}
\label{conjecture:sp2ncenter}
The center of $H_\dfe(\spn_{2n})$ is
 $\mathfrak{z}(H_\dfe(\spn_{2n}))=\mathbb{C}[t_1+C_1, \ldots, t_n+C_n]$ for some $C_i \in \mathfrak{z}(U(\spn_{2n}))$.
\end{conjecture}
\medskip

$\ $

\section{Kostant's Theorem}
\label{section:kostant}
Recall Kostant's theorem in the classical case (\cite{BL}):
\begin{otheorem}
Let $\mathfrak{g}$ be a reductive Lie algebra with an adjoint-type Lie group $G$,
and let $J \subset \mathbb{C}[\mathfrak{g}^*]$ be the ideal generated by the homogeneous elements of $\mathbb{C}[\mathfrak{g}^*]^G$ of positive degree. 
Then:\\
\normalfont{(1)} $U(\mathfrak{g})$ is a free module over its center $\mathfrak{z}(U(\mathfrak{g}))$;\\
\normalfont{(2)} the subscheme of $\mathfrak{g}$ defined by $J$ is a normal reduced irreducible subvariety 
	that corresponds to the set of nilpotent elements in $\mathfrak{g}$.
\end{otheorem}
In \cite{T2}, Kostant's theorem was generalized to $H_\df(\gl_n)$. 
In this section, we provide a similar generalization for $H_\dfe(\spn_{2n})$
assuming Conjecture \ref{conjecture:sp2ncenter}: $\mathfrak{z}(H_\dfe(\spn_{2n}))=\mathbb{C}[t_1+C_1,\ldots,t_n+C_n]$.
As in Section \ref{section:gln}, we define $t_i'=t_i+C_i$.

Introduce a filtration on $H_\dfe(\spn_{2n})$ with $\deg g=1$ for all $g \in \spn_{2n}$ and $\deg v=m+\frac{1}{2}$
for all $v \in V$, where $m$ is half the degree of $\df(z)$; this choice of filtration is also clarified by \cite{LT}.
Let
\[
B_{m}=S(V \oplus \spn_{2n})/\left((-1)^{i-1} \sum_j \{\mathcal{Q}_i,v_j\} v_j^*+c_i^{\mathrm{top}}\right)_{1 \leq i \leq n},
\]
where $\pt_i':=\pt_i+c_i^{\mathrm{top}}$ are the generators of $\mathfrak{z}(H_{r_m}'(\spn_{2n}))$ given in Theorem \ref{theorem:sp2npoissoncenter};
if Conjecture \ref{conjecture:sp2ncenter} is true,
$c_i^{\mathrm{top}}$ is also the highest term of $C_i$.

\begin{theorem}
\normalfont{(1)} Assuming that Conjecture \ref{conjecture:sp2ncenter} is true, $H_\dfe(\spn_{2n})$ is a free module over its center.\\
\normalfont{(2)} $B_m$ is a normal complete-intersection integral domain.
\end{theorem}
\begin{proof}
(1) Introduce a filtration on $B_m$ with $\deg g=1$ for $g \in \spn_{2n}$ and $\deg v=0$ for $v\in V$.
Define $B_m^{(1)}$ by $B_m^{(1)}=\gr B_m=S(V \oplus \spn_{2n})/(c_i^{\mathrm{top}})_{1 \leq i \leq n}$.
The formula in Theorem \ref{theorem:sp2npoissoncenter} implies that $\mathbb{C}[\lambda_1, \ldots, \lambda_n]^{S_n}$
is a free and finite module over $\mathbb{C}[\gr c_1^{\mathrm{top}}, \ldots, \gr c_n^{\mathrm{top}}]$,
so $\mathbb{C}[\mathfrak{h}]^W$ is finite and free over $\mathbb{C}[c_1^{\mathrm{top}}, \ldots, c_n^{\mathrm{top}}]$.
Since $S(\spn_{2n})$ is free over $\mathbb{C}[\mathfrak{h}]^W$ by the classical Kostant's theorem,
$S(\spn_{2n})$, and hence $S(\spn_{2n}) \otimes SV$, is free over $\mathbb{C}[c_1^{\mathrm{top}}, \ldots, c_n^{\mathrm{top}}]$.
Thus,
$S(V \oplus \spn_{2n})$ is free over $\mathbb{C}[\tau_1+c_1^{\mathrm{top}}, \ldots, \tau_n+c_{n}^{\mathrm{top}}]$, 
implying the result.

(2) To show that $B_m$ is a normal integral domain, it suffices to show that the smooth locus 
of the zero set of $\tau_1', \tau_2', \ldots, \tau_n'$
has codimension 2 and is irreducible.
Let $Z=\mathrm{Spec}(B_m)$ be a closed subscheme of $V \oplus \spn_{2n}$
defined by $\tau_i'=0$,
and let
\[
U:= Z \backslash Z_{sm}=\{ (v, A) \in V \oplus \spn_{2n}| (v,A) \in Z \text{ and } \mathrm{rank}(\mathrm{Jac})<n\},
\]
where Jac is the Jacobi matrix of $\tau_1', \tau_2', \ldots, \tau_n'$
at $(v,A)$ with respect to some basis of $V$ and $\spn_{2n}$.
It suffices to show that $U$ is a codimension 2 subvariety of $Z$ and that the latter is irreducible.

Now, recall that
\[
\sum (-1)^{i-1} \{\mathcal{Q}_i, v_j\} v_j^{*} = -(\omega(A^{2i-1}v,v)+\mathcal{Q}_1 \omega(A^{2i-3}v,v)+\mathcal{Q}_2 \omega(A^{2i-5}v,v)+\cdots).
\]
By changing basis, we can rewrite
$((-1)^{i-1} \sum_j \{\mathcal{Q}_i, v_j\} v_j^* + c_i^{\mathrm{top}})_{1 \leq i \leq n}$
as 
${(S_i)}_{1 \leq i \leq n}$, 
where
\[
S_i=-\omega(A^{2i-1} v, v)+f_i(c_1^{\mathrm{top}}, \ldots, c_i^{\mathrm{top}},\mathcal{Q}_1, \ldots, \mathcal{Q}_{i-1})
\]
and $f_i(c_i^{\mathrm{top}}, \mathcal{Q}_i)$ are polynomial expressions in 
$c_1^{\mathrm{top}}, \ldots, c_i^{\mathrm{top}}$ and $\mathcal{Q}_1, \ldots, \mathcal{Q}_{i-1}$ (in particular, there is no dependence on $v$). 
We can and will use the Jacobian of $S_i$ instead of $\tau_i'$ to describe $U$.

Let us calculate the derivatives of $\omega(A^{2i-1} v,v)$ with respect to $y_j \in V$ and $\gamma \in \spn_{2n}$:
\begin{align*}
\frac{\partial}{\partial y_j}(\omega(A^{2i-1} v,v))=2 \omega (A^{2i-1} v, y_j),
\ \ 
\frac{\partial}{\partial \gamma}(\omega(A^{2i-1} v,v))=\omega(A^{2i-2}\gamma v+A^{2i-3}\gamma Av+\cdots+\gamma A^{2i-2} v, v).
\end{align*}
Thus, if
\[
\mu_1 \mathrm{grad}(S_1)+\mu_2 \mathrm{grad}(S_2)+\cdots + \mu_n \mathrm{grad}(S_n)=0
\]
for some $\mu_1, \mu_2, \ldots, \mu_n \in \mathbb{C}$,
then 
\[
\omega(\mu_1 Av, y_j)+\omega(\mu_2 A^3 v, y_j)+\cdots+\omega(\mu_n A^{2n-1} v, y_j)=0
\]
for all $1 \leq j \leq 2n$.
Equivalently, $(\mu_1 A+\mu_2 A^3 + \cdots+\mu_n A^{2n-1})v=0$.

Now we will consider the situation in $B_m^{(1)}=\gr B_m$.
We know that 
$\dim Z = \dim \tilde{Z}$, where $\tilde{Z}=\mathrm{Spec} B^{(1)}_m=V \times \mathcal{N}$ and $\mathcal{N}$ is the nilpotent cone of $\spn_{2n}$.
Since $V$ and $\mathcal{N}$ are irreducible, $\tilde{Z}$, and hence $Z$, is irreducible.
Recall that $U$ was defined as
the locus of points $(v, A) \in Z \subset V \oplus \spn_{2n}$ such that
$\mathrm{rank}(\mathrm{Jac})<n$,
or in other words, all $n \times n$ minors of the Jacobian matrix have determinant 0.
Since each of those determinants is homogeneous with respect to 
our second filtration, it is natural to define $\tilde{U} \subset \tilde{Z}$
as a locus of points where $\mathrm{rank}(\mathrm{Jac})<n$.
Then, $\dim U \leq \dim \tilde{U}$.
Note that $\tilde{U}=\tilde{U}_1 \sqcup \tilde{U}_2$,
where $\tilde{U}_1=\tilde{U} \cap \{(v,A)| A \text{ is regular nilpotent}\}$ and
$\tilde{U}_2=\tilde{U} \cap \{(v,A)| A \text{ is not a regular nilpotent}\}$.
The codimension of a regular nilpotent's orbit is 2, so 
$\mathrm{codim}_{\tilde{Z}}(\tilde{U}_2) \geq 2$.
It suffices to show that $\mathrm{codim}_{\tilde{Z}}(\tilde{U}_1) \geq 2$ as well.
We shall do this by showing that
given a regular nilpotent $A$, $\dim( V_{A, \mathrm{sing}}) \leq 2n-2$,
where $ V_{A, \mathrm{sing}}=\{v \in V| (v,A) \in \tilde {U}\}$.

Let us switch to a basis of $\spn_{2n}$ where
the skew symmetric form is represented by the matrix
\[
J'=
\left(
\begin{matrix}
	0 & \cdots &0       & 0 & -1\\
	0 & \cdots &0      & 1 & 0\\
	0 & \cdots & -1     & 0 & 0\\
	\vdots & \udots & \vdots & \vdots & \vdots \\
	1 & 0      & 0     & \cdots & 0
\end{matrix}
\right).
\]
If we define
\[
A=
\left(
\begin{matrix}
	0 & 1 & 0 & \cdots & 0\\
	0 & 0 & 1 & \cdots & 0\\
	\vdots & \vdots & \cdots & \ddots & 0\\
	0 & 0 & 0 & \cdots & 1 \\
	0 & 0 & 0 & \cdots & 0
\end{matrix}
\right),
\]
then $AJ'+J'A^T=0$, implying $A \in \spn_{2n}$.
Now, suppose that $\sum_{1 \leq j \leq n} \mu_j \mathrm{grad}(S_j)=0$ at $(A,v)$, for $v=(a_1, \ldots, a_{2n})$.
By examining the $\frac{\partial}{\partial y_j}$ components of $\mathrm{grad}(S_j)$, we get $a_{2n}=0$;
moreover, either $a_{2n-1}=0$, or $\mu_1=\cdots=\mu_{n-1}=0$.
The conditions $a_{2n}=a_{2n-1}=0$ define a codimension two subspace as desired.
We thus need to show that if $a_{2n}=0$ and 
$\mu_1=\cdots=\mu_{n-1}=0$,
then 
$\sum_{1 \leq j \leq n} \mu_j \mathrm{grad}(S_j)=0$
implies
a nontrivial condition on $v$.
To find such a condition, note that 
\[
\frac{\partial}{\partial\gamma}(\omega(A^{2n-1}v,v))=\omega(A^{2n-2}\gamma v,v)+\omega(A^{2n-3} \gamma Av,v)+\cdots+\omega(\gamma A^{2n-2}v,v),
\]
and that $\frac{\partial}{\partial\gamma}f_i(c_1^{\mathrm{top}}, \ldots, c_i^{\mathrm{top}},\mathcal{Q}_1, \ldots, \mathcal{Q}_{i-1})$ does not depend on $v$. 
Now, let us take $\gamma=e_{2n,1}$;
we can verify that $e_{2n,1}J'+J'e_{2n,1}^T=0$, so $e_{2n,1} \in \spn_{2n}$.
We note that $e_{2n,1} A^{2n-2}=e_{2n,2n-1}$, $Ae_{2n,1}A^{2n-3}=e_{2n-1,2n-2}$, 
$A^2e_{2n,1}A^{2n-4}=e_{2n-2,2n-3}$
and so forth.
Thus, $\frac{\partial}{\partial\gamma}(\omega(A^{2n-1}v,v))=\omega(A^T v, v)$.
However, setting $v=(a_1, \ldots, a_{2n-1}, 0)$, we get $\omega(A^T v,v)=\omega((0,a_1,\ldots,a_{2n-1}),(a_1,\ldots,a_{2n-1},0))$,
which is a nontrivial degree two polynomial in $a_1, \ldots, a_{2n-1}$
that should equal the number 
$\frac{\partial}{\partial\gamma}(f_i(c_1^{\mathrm{top}}, \ldots, c_i^{\mathrm{top}},\mathcal{Q}_1, \ldots, \mathcal{Q}_{i-1}))(A)$. 
This gives the other codimension 1 condition, and so
$\tilde{U}_1$ is at least of codimension 2 in $\tilde{Z}$ as desired.
\end{proof}

$\ $

\appendix
\section*{Appendix: Proof of Lemma \ref{lemma:sp2nkey}}
In this section, we will outline the proof of  Lemma \ref{lemma:sp2nkey}, which states:
\begin{equation}
\tag{$\dagger$}
\label{eqn:identity}
\sum_{j=1}^{2n} \sum_{e\in B}\left\{\frac{\partial \mathcal{Q}_i}{\partial e} ,v \right \} e(v_j)v_j^*=0.
\end{equation}
We use the basis for $V$ defined in Section \ref{section:sp2n}, in which $\omega$ is represented by the matrix $J$.

Let us multiply (\ref{eqn:identity}) by $t^{2i}$ and sum over $i$ to get the equivalent assertion that
\[
\sum_j \sum_{e\in B}\left \{\frac{\partial \det(1-tA)}{\partial e} ,v \right\} e(v_j)v_j^*=0.
\]
Since the whole sum is $\mathfrak{sp}_{2n}$-invariant (even though each term considered separately is not),
we can look at the restriction of the sum to $\mathfrak{h}$.
Thus, this sum equals zero if and only if
\[
\left. \sum_j\sum_{e\in B}\left \{\frac{\partial \det(1-tA)}{\partial e} ,v\right\} e(v_j)v_j^* \right|_{\mathfrak{h}}=0.
\]

We choose the following basis $B$
for $\spn_{2n}$: $e_{2j-1, 2j}$, $e_{2j, 2j-1}$, $e_{2j-1,2j-1}-e_{2j,2j}$, for all $1 \leq j \leq n$,
and for all $1 \leq k < l \leq n$, the elements $e_{2l-1,2k}+e_{2k-1,2l}$, $e_{2l,2k}-e_{2k-1,2l-1}$, $e_{2l-1,2k-1}-e_{2k,2l}$, and
$e_{2l,2k-1}+e_{2k,2l-1}$.
We observe that for any $1\leq j, j' \leq 2n$,
there exists a unique basis vector in $B$ that takes $v_{j}$ to $\pm v_{j'}$;
we shall denote this element by $v_{j',j} \in B$.
These $v_{j', j}$ are not pairwise distinct
since there are basis vectors with two nonzero entries.

Since $\mathrm{Sp}_{2n}$ acts transitively on $V$, we can assume $v=v_1$. Using the above basis, we get
\[
\sum_j \sum_{e\in B}\left\{\frac{\partial \det(1-tA)}{\partial e} ,v_1 \right\} e(v_j)v_j^*=
\sum_{j,j',k} \frac{\partial^2 \det(1-tA)}{\partial v_{k,1} \partial v_{j',j}} v_{j'} v_k v_j^* (-1)^{\iota_{jj'}},
\]
where
\begin{displaymath}
   \iota_{jj'} = \left\{
     \begin{array}{ll}
       1 & \text{if } j \equiv j' \text{ mod 2 and } j<j'\text{, or if } j'=j \text{ and }j\text{ is even,}\\
       0 & \text{otherwise.}
     \end{array}
   \right.
\end{displaymath} 
We now restrict to $\mathfrak{h}$.
We have $\left. \frac{\partial^2 \det(1-tA)}{\partial v_{k,1}\partial v_{j',j}} \right|_{\mathfrak{h}} \neq 0$ 
only when the matrices for $v_{k,1}$ and $v_{j',j}$ have nonzero entries on the diagonal, or
if $v_{k,1}$ and $v_{j',j}$ have nonzero entries at the $i$-th row $j$-th column and $j$-th row $i$-th column, respectively.
This can only happen when $v_{j'} v_k v_j^*=v_1 v_a v_a^*$ for some $a$.
We can list all the ways this can happen for $a=2b$ or $a=2b-1$ with $b \neq 1$ (keeping in mind that $v_{2b-1}^*=v_{2b}$ and $v_{2b}^*=-v_{2b-1}$):
\begin{enumerate}
	\item $\frac{\partial^2 \det(1-tA)}{\partial v_{1,1} \partial v_{2b-1,2b-1}} v_{1} v_{2b-1} v_{2b}$,
	\item $\frac{\partial^2 \det(1-tA)}{\partial v_{1,1} \partial v_{2b,2b}} v_{2b} v_1 v_{2b-1}$,
	\item $\frac{\partial^2 \det(1-tA)}{\partial v_{2b-1,1} \partial v_{1,2b-1}} v_{1} v_{2b-1} v_{2b}$,
	\item $\frac{\partial^2 \det(1-tA)}{\partial v_{2b,1} \partial v_{1,2b}}( -v_{1} v_{2b} v_{2b-1})$,
	\item $\frac{\partial^2 \det(1-tA)}{\partial v_{2b,1} \partial v_{2b-1,2}} (-v_{2b-1} v_{2b} v_1)$,
	\item $\frac{\partial^2 \det(1-tA)}{\partial v_{2b-1,1} \partial v_{2b,2}} v_{2b-1} v_{2b} v_1$.
\end{enumerate}
To calculate the derivatives,
let $A_1$ be the 4 by 4 matrix formed by the intersections of the first, second, $2b-1$-th, and $2b$-th rows and columns of $A$,
and let $A_2$ be the $2n-4$ by $2n-4$ matrix formed by the intersections of the remaining rows and columns.
The space of all such $A_2$ is isomorphic to $\spn_{2n-4}$,
and we denote the Cartan subalgebra of diagonal matrices of this
space by $\mathfrak{h}(A_2)$.
All six of the above derivatives evaluate to the same polynomial in $\mathfrak{h}(A_2)$
times the corresponding derivative in $\spn_4$; for instance,
$\frac{\partial^2 \det(1-tA)}{\partial v_{1,1} \partial v_{2b-1,2b-1}}=h\frac{\partial^2 \det (1-tA_1)}{\partial v_{1,1}'\partial v_{3,3}'}$ with $v_{1,1}', v_{3,3}' \in \spn_4$ and $h \in S(\mathfrak{h}(A_2))[t]$.
Thus, we can reduce our problem to $\spn_4$, and straightforward computations verify (\ref{eqn:identity}) for $\spn_4$.
Similarly, when $b=1$ (that is, when the term is of the form $v_{1} v_{1} v_{2}$), all computations will reduce to analogous ones in $\spn_2$.

$\ $

{\bfseries Acknowledgments:}
The authors would like to thank Pavel Etingof for
suggesting this research topic,
for stimulating discussions,
and for reviewing the rough draft of this paper.
The authors would also like to thank the PRIMES program at MIT for sponsoring this research.
Finally, the authors would like to thank the referee for 
pointing out some inaccuracies in the original proof of
Theorem \ref{theorem:classification}.

$\ $

\end{document}